\documentclass[11pt,reqno,twoside]{amsart}

\usepackage{amsmath}
\usepackage{amsfonts}
\usepackage{amssymb}

\usepackage{epsfig}

\usepackage{xcolor}

\numberwithin{equation}{section}

\newif\ifdraft\drafttrue

\draftfalse

\newcommand{\m}{{m}}

\newcommand\mr{M_{m,n}}

\newcommand\hd{Hausdorff dimension}

\newcommand\da{Diophantine approximation}

\newcommand\ssm{\smallsetminus}

\newcommand\eq[2]{{\ifdraft{\ \tt [#1]}\else\ignorespaces\fi}\begin{equation}\label{eq:#1}{#2}\end{equation}}
\newcommand {\equ}[1]     {\eqref{eq:#1}}

\newcommand{\R}{{\mathbb {R}}}
\newcommand{\Z}{{\mathbb {Z}}}
\newcommand{\N}{{\mathbb {N}}}

\newcommand{\Leb}{{\operatorname{Leb}}}

\newcommand{\Ad}{{\operatorname{Ad}}}
\newcommand{\ad}{{\operatorname{ad}}}

\newcommand{\Tr}{\operatorname{Tr}}
\newcommand{\SL}{\operatorname{SL}}
\newcommand{\PSL}{\operatorname{PSL}}
\newcommand{\ggm}{G/\Gamma}

\newcommand{\diam}{\operatorname{diam}}
\newcommand{\dist}{\operatorname{dist}}
\newcommand{\diag}{\operatorname{diag}}
\newcommand{\supp}{\operatorname{supp}}

\newcommand{\Lie}{\operatorname{Lie}}

\newcommand {\ignore}[1]  {}

\newcommand\hs{homogeneous space}

\newcommand\cic{C^\infty_{comp}}

\newcommand{\df}{{\, \stackrel{\mathrm{def}}{=}\, }}

\newcommand{\vv}{{\bf{v}}}
\newcommand{\vs}{{\bf{j}}}

\newcommand{\vr}{{\bf i}}

\newcommand{\p}{{\bf p}}

\newcommand{\vq}{{\bf q}}

\newcommand {\comm}[1]   {\textcolor{red}{#1}}

\newcommand {\gr}[1]     {\textcolor{green}{#1}}

\DeclareMathOperator{\codim}{codim}
\DeclareMathOperator{\spn}{span}

\newcommand{\vre}{\varepsilon}

\newcommand\nz{\smallsetminus \{0\}}

\newtheorem{thm}{Theorem}[section]
\newtheorem{lem}[thm]{Lemma}
\newtheorem{prop}[thm]{Proposition}
\newtheorem{cor}[thm]{Corollary}

\newtheorem{defn}[thm]{Definition}

\newtheorem{claim}[thm]{Claim}

\begin{document}
\title[Dimension drop   for diagonalizable flows]
{{Dimension drop for diagonalizable \\ flows on homogeneous spaces}}
\author{Dmitry Kleinbock}
\address{Department of Mathematics, Brandeis University, Waltham MA}
\email{kleinboc@brandeis.edu}

\author{Shahriar Mirzadeh}
\address{{ Department of Mathematics, University of Cincinnati, Cincinnati OH }}
 \email{{mirzadsr@ucmail.uc.edu}}

\begin{abstract} 
Let $X = \ggm$, where $G$ is a Lie group and $\Gamma$ is a lattice in $G$,  let $O$ be an open subset of $X$, {and let $F = \{g_t: t\ge 0\}$ be a one-parameter subsemigroup of $G$.  
Consider the set of points in $X$ whose $F$-orbit misses $O$; it has measure zero if the flow is ergodic. It has been conjectured that this set has Hausdorff dimension strictly smaller than the dimension of $X$. This conjecture is proved when $X$ is compact or when $G$ is a simple Lie group of real rank $1$, or, most recently, {for certain flows on 
the space of lattices.}
In this paper we  prove this conjecture 
for arbitrary $\Ad$-diagonalizable flows on irreducible quotients of semisimple Lie groups. 
The proof uses exponential mixing of the flow together with the  method of integral inequalities for height functions on $\ggm$. We also {derive} an application to jointly Dirichlet-Improvable systems of linear forms.}
\end{abstract}

\thanks{The first-named author was supported by NSF grant  DMS-1900560.}

\date{July 2022}

\maketitle
\section{Introduction}
\subsection{The set-up} {{Let  $G$ be a {connected} 
{Lie group, 
and let $\Gamma $   
be a lattice} in $G$.} 
Denote by $X$ the homogeneous space  $G/\Gamma$ and by $\mu$   the
$G$-invariant probability measure on $X$. 
%
%
For {an unbounded subset $F$ of $G$
and a non-empty open subset $O$ of $X$  define the sets ${E(F,O)}$  and $\widetilde E(F,O)$ as follows:
\eq{set} {\begin{aligned}E(F,O) &:=  \{ x \in X: gx \notin
 O \ \forall\, g\in F\}\\\subset\  \widetilde E(F,O)&:= \{ x \in X: \exists\text{ compact }Q\subset G\text{ such that }gx \notin
 O \ \forall\, g\in F\ssm Q\}\\ &= \bigcup_{\text{compact }Q\subset G}E(F\ssm Q,O)
 \end{aligned}
 }}
of points in $X$ whose 
{$F$}-trajectory always (resp., eventually) stays away from $O$.}
If {$F$ is a subgroup or a subsemigroup of $G$ acting ergodically} 
on $(X,{\mu})$,
then the set {$\{gx : g\in F\}$} is dense for {$\mu$-}almost all $x \in X$, in particular $\mu\big(\widetilde E({F},O)\big) = 0
$.


{The present paper studies the following natural question, asked several years ago by Mirzakhani: 
for a 
subgroup or sub-semigroup $F\subset G$, {if  the set $E(F,O)$ has measure zero, does it necessarily have less than} full \hd?} it is reasonable to conjecture that the answer is always affirmative; in other words, that the following `Dimension Drop Conjecture' (DDC) holds: {\it if 
$F\subset G$ is a subsemigroup and $O$ is an open subset of $X$, then either $E(F,O)$ has positive measure, or its dimension is less than the dimension of $X$.}
When $X$ is compact it follows from the variational principle for measure-theoretic entropy, as outlined in \cite[\S 7]{KW2}; an effective argument using exponential mixing was developed in \cite{KMi1}. See also \cite[Theorem 1.1 and Corollary 1.3]{MRC} which explores the  dimension drop phenomenon in a different setting.

In the non-compact case a weaker statement that  for any non-quasiunipotent flow on a finite volume homogeneous space, the set   of points
that lie on  {\it divergent} trajectories  has positive codimension is a conjecture made by Cheung in \cite{C}.
A standard approach to this circle of problems is to use  the phenomenon of {\it non-escape of mass} on \hs s, going back to the work of Eskin--Margulis--Mozes and Eskin--Margulis, see \cite{EMM, EM} and also \cite{KKLM}. This is precisely how Cheung's conjecture has been recently verified by Guan and Shi \cite{GS}; see also \cite{AGMS, RHW} for some related work.
However combining the  non-escape of mass argument with an additional construction taking care of the compact part of the space is more involved. Previously this was done in the case when {$G$ is a simple Lie group of} real rank $1$ \cite{EKP}, and then, in the most recent work of the authors \cite{KMi2}, 
{when} 
\eq{gt}{
\begin{aligned}X = \SL_{m+n}(\R)/\SL_{m+n}(\Z) 
\text{ and } \qquad\quad\\ F = \big\{
\diag (e^{nt}, \dots, e^{nt},e^{-mt}, \dots, e^{-mt}): t\ge 0\big\}.\end{aligned}  }

\smallskip
In this paper we generalize the approach of \cite{KMi2} 
by exhibiting two abstract assumptions sufficient for the validity of DDC. One takes care of the compact part of the space, while the other deals with the non-escape of mass.

Let $G$ be a 
Lie group, $\Gamma$ a lattice in $G$ and  $X = \ggm$. 
We shall start by introducing some notation. 
Fix a {right-invariant} Riemannian {structure  on $G$, and denote by
`${\dist}$' the corresponding Riemannian metric},
  {using} the same notation for the induced metric on homogeneous spaces of $G$. 
  {In what follows, if $P$ is a subgroup of $G$, we will denote by $B^{P}(r)$  the open ball {in $P$} of radius $r$ centered at the identity element with respect to the  
metric  on $P$ corresponding to the Riemannian structure induced from $G$.}
{We will let $\nu$ stand for the Haar measure on $P$ normalized so that $\nu\big(B^{P}(r)\big) = 1$.}  
{For simplicity, we use $B(r)$ instead of  $B^G(r)$ to denote a ball of radius $r$ in $G$ centered at the identity element. Also,} $B(x,\rho)$ will stand for the open ball  in $X$ centered at $x\in X$ of radius $\rho$.

For $x\in X$ denote by $\pi_x$ the map  $G\to X$ given by $\pi_x(g) := gx$, and by  $r_0(x)$ the \textsl{injectivity radius} of $x$: $$
{r_0(x) :=}\,\sup\{r > 0: 
\pi_x\text{ is injective on }B(r)\}.$$
If $K$ is a subset of $X$, let us denote by $r_0(K)$ the \textsl{injectivity radius} of $K$: $$
r_0(K) := \inf_{x\in K}r_0(x) = \sup\{r > 0: 
\pi_x\text{ is injective on }B(r)\  \ \forall\,x\in K\};$$
it is known that $r_0(K) > 0$ if and only if $K$ is bounded.

The notation ${A\gg B}$
where  $A$ and $B$ are quantities depending on certain parameters, will mean ${A \ge  CB}$, 
where {$C$ is a constant} 
independent   on those parameters.

Let now $F = \{g_t: t \ge 0\}$ be an $\Ad$-diagonalizable one-parameter subsemigroup of $G$.
 A key role in our method will be played by the  {\sl unstable horospherical subgroup} with respect to $F$, defined as
\eq{uhs}{H := \{ g \in G:{\dist}({g_t}g{g_{ - t}},e) \to 0\,\,\,as\,\,\,t \to  - \infty \}.}
Equivalently, $H$ is the Lie group whose Lie algebra is a direct sum of eigenspaces of $\Ad\, g_1$ corresponding to eigenvalues with absolute value $> 1$. More generally, we will work with {connected} subgroups $P$ of $H$ normalized by $F$  and will give conditions sufficient for `dimension drop along $P$-orbits'; that is, ensuring a nontrivial upper estimate for 
$$ \dim \big(\{h\in P: hx\in \widetilde E({F},O)\} \big),$$ where $x\in X$ is arbitrary, and $O$ is a non-empty open subset of $X$.

Throughout the proof we will pay close attention to translates of $P$-orbits in $X$ by $g_t$.  it will be convenient to use the following notation: if 
 $f$ a function on $P$
 and $t\ge 0$, we will define the integral operator 
$I_{f,t}$ acting on functions $\psi$ on $X$ via
\eq{ift}{(I_{f,t}\psi)(x):= \int_P
f({{h}})\psi(g_t{{h}}x)\,d{{\nu}}({{h}})\,.}
In other words, $(I_{f,t}\psi)(x)$ is the integral of $\psi$ with respect to the $g_t$-translate of the $\pi_x$-pushforward of the signed measure $f\,d{{\nu}}$. When $f = 1_B$ for a subset $B$ of $P$, we will write $$I_{B,t} = \int_B
 \psi(g_t{{h}}x)\,d{{\nu}}({{h}})$$ in place of $I_{1_B,t}$.

\subsection{Exponential mixing and effective equidistribution} 
The first ingredient of our proof is the 
effective equidistribution of $g_t$-translates of $P$-orbits on $X$. 
To {introduce} this property we will work with Sobolev spaces of functions on $X$. 
{Let us define}
$$C^\infty_2(X) = \big\{h \in C^\infty(X): \|h \|_{\ell{,2}} < \infty\text{ for any }\ell \in \Z_+\big\},$$
where $\|\cdot \|_{\ell{,2}}$ is the ``{\sl $L^{{2}}$, order $\ell$" Sobolev norm} {(see \S\ref{tess} for more detail). Now} 
let us introduce the following 

\begin{defn}\label{subgroup} \cite{KMi1}
Say that a subgroup   $P$ of $G$ has  \textsl{Effective Equidistribution
Property}  {\rm (EEP)} with respect to the flow $(X,F)$
if 
 there exists constants {$a,b,
 \lambda > 0$} and {$\ell \in {\N}$} such that 
for any 
{$x \in X$ and $t > 0$ with \eq{conditionont}{t\  {\ge 
a+b \log\frac{1}{r_0({x}) } 
 ,} }
 any $f\in C^\infty(P)$ with $\supp f \subset B^P(1)$ and 
any $\psi\in
C^\infty_2(X)$}
it holds that
\eq{eep}{\left| {(I_{f,t}\psi)(x) - \int_P f\,d\nu {\mkern 1mu} \int_X \psi\,d\mu  {\mkern 1mu} } \right| {\ll}\ {\max \big(
{\left\| \psi  \right\|_{C^1}, \left\| \psi  \right\|_{\ell{,2}} }
\big)} \cdot {\left\| f \right\|_{{C^\ell }}} \cdot {e^{ - \lambda t}}{\mkern 1mu}.}
\end{defn}
{Note that the constants $a,b$ in \equ{conditionont} and an implicit constant in \equ{eep}} are allowed to depend on $P$ and $F$, but not on $x,t,f$ and $\psi$.

A general principle that mixing implies equidistribution of unstable leaves---that is, orbits of $H$ as in \equ{uhs}---has been widely used in homogeneous dynamics, starting perhaps with the Ph.D.\ thesis of Margulis \cite{M}. 
Its effective versions have been exploited in \cite{KM1, KM4}.
\ignore{More recently, in \cite{KMi1} the authors introduced  the {\sl Effective
Equidistribution Property}  of subgroups of $H$ with respect to  the flow $(X,F^+)$. 
Roughly speaking,  it
 asserts that for functions $\psi$   on $X$ and $f$  on $P$, the quantity $I_{f,\psi}({t},x)$ tends to the product of integrals of $f$ and $\psi$ as $t\to\infty$. Moreover, the error term can be effectively estimated, which can happen only when some restrictions on $x$ are placed.}
 Specifically, 
let us say that a flow $(X,F)$ is {\sl exponentially
mixing}  if there exist {$\gamma  > 0$ and $\ell\in\Z_+$}
such that for any   $\varphi, \psi \in  C^\infty_2(X) $
 and for any $t\ge 0$ one has
$${\left| {({g_t}\varphi ,\psi )- \int_X\varphi\, d\mu \int_X\psi\, d\mu}\right|   \ll  {e^{ - \gamma t}}{\left\| \varphi  \right\|_{\ell{,2}}} {\left\| \psi  \right\|_{\ell{,2}}}.}$$
This property for non-quasiunipotent flows   follows from the {{\sl strong spectral gap} of the regular representation of $G$, see  \cite{KM1}; the latter is known to hold for 
quotients of semisimple Lie groups without compact factors by irreducible lattices, see \cite[p.~285]{KS}.}

{The fact that property (EEP) for expanding horospherical subgroups follows from exponential mixing was established in \cite{KMi1} by a variation of the method developed in \cite{KM1}:}

\begin{thm}\label{thmheep}\cite[Theorem 2.5]{KMi1}  Let $G$ be a Lie group, $\Gamma $   a 
lattice in $G$, and let ${F}$ be a one-parameter 
{sub}semigroup of $G$ 
whose action on $X = G/\Gamma$  is exponentially mixing. Then $H$ as in \equ{uhs} satisfies property {\rm (EEP) {with respect to the flow $(X,F)$}}.\end{thm}

\ignore{Another example is given by  $X =\SL_{m+n}(\R)/\SL_{m+n}(\Z)$ 
and
\eq{weights}{g_t = {g_t^{\vr,\vs} :=} \diag({e^{{i_1}t}},\dots,{e^{{i_m}t}},{e^{
- {j_1}t}},\dots,{e^{ - {j_n}t}}),}
where
$$\vr = ({i_k}: k = 1,\dots, m)\text{ and }\vs = ({j_\ell}:  \ell = 1,\dots,  n)$$
{and 
$${{i_k},{j_\ell} > 0\,\,\,\,and\,\,\,\,\sum\limits_{k = 1}^m {{i_k} = 1 = \sum\limits_{\ell = 1}^n {{j_\ell}} }.}$$}
This set-up is important for simultaneous \da\ with weights \comm{(will we have a corollary about it in this paper? if not, maybe we can remove this example)}.
it is shown in \cite[Theorem 7.1]{KMi1}
that the subgroup
\eq{weightsp}
{P = \left\{ \left( {\begin{array}{*{20}{c}}
{{I_m}}&A\\
0&{{I_n}}
\end{array}} \right) : A\in  {M_{m,n}}(\mathbb{R})\right\}
}
satisfies (EEP) relative to the
$g_t^{\vr,\vs}$-action.}

\subsection{Height functions and integral inequalities} The second ingredient of our proof 
is studying $g_t$-translates of $P$-orbits in $X$ at infinity. For that it is helpful to have a family of positive functions on $X$ which grow at infinity and behave nicely with respect to integral operators of type \equ{ift}
This is done via the method of integral inequalities which goes back to \cite{EMM} and \cite{EM}, and was recently applied in  \cite{GS} 
 for upper estimates for the \hd\ of the set of points of $X$ with divergent $g_t$-trajectories. To state their result, it will be convenient to introduce certain terminology, which will be used throughout the paper.   Namely, let us say that a non-negatve continuous function $u$ on $X$  is a
 \textsl{height function} if it is proper, that is $u(x) \to\infty$ if and only if $x \to\infty$ in $X$, and  \textsl{regular}, that is there exists a non-empty neighborhood $B$ of identity in $G$ and $C > 0$ such that \eq{reg}{u(hx)\le Cu(x)\text{ for every   $h\in B$ and all }x\in X;} equivalently, if for any bounded $B\subset G$ there exists $C > 0$ such that  \equ{reg} holds. 
 Also let us say that $u$  satisfies the \textsl{$(c,d)$-Margulis inequality} with respect to an operator $I: C(X)\to C(X)$ if for all $x\in X$ one has
$$
(Iu)(x) \le cu(x) + d.
$$
 {See \cite{EMo, SS}, where functions $u$ satisfying the  $(c,d)$-Margulis inequality for some $c<1$ and $d\in\R$ are called {\sl Margulis functions}.} 
 With this terminology, 
 let us introduce the following definition.

 \begin{defn}\label{enp2}
Say that a subgroup   $P$ of $G$ has  \textsl{Effective Non-Divergence
Property } {\rm (ENDP)} with respect to the flow $(X,F)$ if 
{there exists $0<c_0<1$ and $t_0>0$ such that 
for any 
$t \ge t_0$
one can find $d_t>0$ and  a height function $u_t$
such that $u_t$ satisfies the  $(c_0,d_t)$-Margulis inequality {with respect to $I_{B^P(1),t}$}.}
\end{defn}

{In the course of proving the main result of \cite{KMi2} the above property was shown   in the case \equ{gt}, see \cite[Proposition 3.4.]{KMi2}. The proof followed a construction fron \cite{KKLM} and used   functions on the space of lattices coming from the work of  Eskin, Margulis and Mozes  \cite{EMM}. To get more examples,} we will quote the following result from \cite{GS}:
\begin{thm}\label{thmguanshi}  \cite[Lemma 4.3]{GS} 
Let $X 
= \prod_{
i=1}^n G_i/\Gamma_i$, where each $G_i/\Gamma_i$
is a non-uniform irreducible
quotient of a semisimple Lie group without compact factors, and let $F$
be be a one-parameter {$\Ad$-}diagonalizable
{sub}semigroup of $G= \prod_{
i=1}^n G_i$ such that the projection of $F$ to each $G_i$ is unbounded. Also let $H$ be as in \equ{uhs} and let $B = B^H(1)$. Then there exist $\alpha, t_0 >  0$ and, for any 
 $t \ge t_0$, {a height function $u_t$ on $X$ and $d_t\in \R$ such that
for any $t \ge t_0$ the function $u_t$  satisfies an  $(e^{-\alpha t},d_t)$-Margulis inequality with respect to $I_{B,t}$}. {Consequently, $H$ 
has property {\rm (ENDP)} with respect to $(X, F)$.}
\end{thm}
\ignore{The above theorem immediately implies the following corollary
\begin{cor}
Let $X$ be as in Theorem \ref{thmguanshi} and $F^+$ be be a one-parameter {$\Ad$-}diagonalizable
{sub}semigroup of $G$ such that the projection of $F^+$ to each $G_i$ is unbounded. Then $H$ as in \equ{uhs} has property {\rm (ENDP)} with respect to $(X, F)$.

\end{cor}

}

{We remark that, since the flow $(X,F)$ as in the above theorem is  exponentially mixing, the expanding horospherical subgroup $H$ has property (EEP) with respect to $(X, F)$ as well.}

\subsection{The main results} We are now ready to state {our} main theorem. 

\smallskip
\ignore{In what follows, the notation ${A\gg B}$,
where  $A$ and $B$ are quantities depending on certain parameters, will mean ${A \ge  CB}$,  
with {$C$ being a constant} 
dependent  only on $X$ and $F^+$. }

\ignore{\begin{thm}\label{dimension drop 2} Let $G=\SL_{m+n}(\R)$, $\Gamma= \SL_{m+n}(\Z)$, fix $T > 0$  and let $g = g{T}$ where $g_t$ is as in \eq{gt}. Then 
here exist positive constants $K_P, p' \ge 1$, and $r_3 \le 1 $ such that for any open subset $O$ of $\ggm$ and any $r$ satisfying $0<r<\min( \mu (\sigma_r O)^{p'}, r_3)$, the set $E(g,O)$ has Hausdorff codimension at least
$$ K_P \frac{\mu(\sigma_r (O ))}{\log \frac{1}{r}}. $$ 
\end{thm} }

\ignore{\begin{thm} \label{delta average}
There exist positive constants $K_P, p' \ge 1$, and $r_3 \le 1 $ independent of $O$ such that for any $0 < \delta \le 1$, and for any $r$ satisfying $0<r<\min( \mu (\sigma_r O)^{p'}, {{r_{2}}})$, the set$$\{ x \in X:x \,\,\, \delta \textsl{-escapes} \,\,O \,\, \textsl{on average}  \}$$ has Hausdorff codimension at least
$$ K_P \delta \frac{\mu(\sigma_r (O ))}{\log \frac{1}{r}}. $$ 
\end{thm}}

{

{{\begin{thm}\label{dimension drop 3} 
Let $G$ be a Lie group, $\Gamma $   a 
lattice in $G$, $X = G/\Gamma$,  
$F$   a one-parameter {$\Ad$-}diagonalizable
{sub}semigroup of $G$, $H$ the unstable horospherical subgroup relative to $F$, and 
 $P$   a subgroup of $H$ which is normalized by $F$ and  has properties {\rm (EEP)} and {\rm (ENDP)} with respect to the flow $(X,F)$. Then 
{for any non-empty open subset $O$ of $X$ one has}
$${\inf_{x \in X} \codim \left(\big\{h\in P: hx\in \widetilde E({F},O)\big\} \right)  >0. }$$  
  \end{thm} }}

Applying the above theorem with $P = H$ and using {Theorems \ref{thmheep} and \ref{thmguanshi}} together with the standard slicing technique, we get
{\begin{cor}\label{dimension drop 2} 
{Let $G, \Gamma , X$ and ${F}$ 
be as in Theorem \ref {thmguanshi}.}
 Then for any non-empty open subset $O$ of $X$ one has
$ \dim \widetilde E({F},O) < \dim X      $; {that is, DDC holds in this generality}.
\end{cor} 

{We remark that the main result of \cite{KMi2} in the case \equ{gt} actually contains an effective upper bound for the dimension of $E({F},O)$. In the more general set-up of this paper it is also possible to make our estimates effective. {For that one needs to find a lower bound for the injectivity radius of compact sets $\{ x: u_t(x) \le M   \}$. Such bounds can be obtained for a large class of height functions including $u_t$ as in Theorem \ref {thmguanshi}  using a similar procedure as in \cite[Proposition 26]{SS} and \cite[Lemma 6.3]{BQ}.}
We have decided not to overcomplicate the exposition with the proof of the stronger result;  however see \S\ref{remarks} for some indications of the proof.}

{Another remark is that, similarly to \cite{KMi2}, we could have considered cyclic semigroups $F$ of the form {$\{g^t :  t\in\Z_+\}$}, where $g$ is an $\Ad$-diagonalizable element of $G$. Then, after replacing $g_t$ with $g^t$ in \equ{ift}, the conclusions of Theorem \ref{dimension drop 3} and Corollary \ref{dimension drop 2}  can be established for discrete-time actions, with minor modifications of the proofs.}
\smallskip

{The structure of the paper is as follows. In the next section we state a technical  theorem (Theorem \ref{first}) and show how it implies  Theorem \ref{dimension drop 3} and Corollary \ref{dimension drop 2}. 
The proof of Theorem \ref{first} occupies the bulk of the paper. It has two main ingredients: one deals with orbits staying inside a fixed compact subset of $X$, which are handled in \S\S\ref{measest}--\ref{boxes} with the help of the effective equidistribution assumption. 
The other one (\S\S \ref{iterations}--\ref{escape}) takes care of orbits venturing far away into the cusp of $X$; there we use property (ENDP) via the method of integral inequalities for height functions on $X$. The two ingredients are combined in  {\S\S\ref{abstractlemma}--\ref{endofproof}.}
Some concluding remarks, including an application to joint Dirichlet improvement in \da, are presented in \S\ref{remarks}.}

\smallskip

{\noindent
\textbf{Acknowledgements.}
The authors are grateful to Alex Eskin for bringing  Mirzakhani's question to their attention, and to  Victor Beresnevich  for useful remarks, in particular for asking a question that led to Theorem \ref{jdi}.}


\ignore{It what follows, if $P$ is a subgroup of $G$, we will denote by $B^{P}(r)$  the open ball of radius $r$ centered at the identity element with respect to the  
metric  on $P$ coming from the Riemannian structure induced from $G$. In particular, we have $$B^H(r) = \{h_s: s\in M_{m,n},\ \|s\|\le r\}.$$ The ball $B^G(r)$ in $G$ will be simply denoted by 
$B(r)$,
and $B(x,r)$ will denote  the open ball of radius $r$ centered at $x \in X$. Given  $x\in X$, let us denote by $\pi_x$ the map  $G\to X$ given by {$\pi_x(h) = hx$}, and by  $r_0(x)$ the \textsl{injectivity radius} of $x$, defined as $$
\sup\{r > 0: 
\pi_x\text{ is injective on }B(r)\}.$$
If $K\subset X$ is bounded, let us denote by $r_0(K)$ the \textsl{injectivity radius} of $K$: $$
r_0(K) := \inf_{x\in K}r_0(x) = \sup\{r > 0: 
\pi_x\text{ is injective on }B(r)\  \ \forall\,x\in K\}.$$}
{
  \smallskip

\section{{Theorem \ref{first} $\Rightarrow$ Theorem \ref{dimension drop 3} $\Rightarrow$ 
Corollary \ref{dimension drop 2} 
}}\label{easyproofs}
   Let us introduce the following notation: for a {non-empty} open subset $O$ of $X$ and $r > 0$ denote by $\sigma_rO$ the \textsl{inner $r$-core\/} of $O$, defined as
\eq{inner-core}
{\sigma_r O :=  \{x\in X: {\dist}(x,O^c) > r\}.}
{This is an open subset of $O$, {whose measure is close to $\mu(O)$ for small enough values of $r$. 
\smallskip}

{{Furthermore}, for a closed subset $S$ of $X$ denote  by  $\partial_rS$ the \textsl{$r$-neighborhood\/}   of $S$, that is, 
 $$
\partial_rS :=  \{x\in X: {\dist}(x,S)  < r\}.
$$
 In particular, for $z\in X$ we have $\partial_r\{z\} = B(z,r)$,   the open ball of radius $r$ centered at $z$.
 Note that we always have \eq{boundarysigma}{\partial_rS \subset \big(\sigma_r(S^c)\big)^c\text{ for all }S\subset X,\ r > 0.}
{Also, given $G$, $H$, and  $F = \{g_t: t\ge 0\}$ as in Theorem \ref{dimension drop 3}, for any subgroup $P$ of $H$ that is normalized by $F$ define:  
\eq{b1}{\lambda_{\min} := \min \big\{ \lambda :\,\lambda \text{ is an eigenvalue of }\ad_{{g_1}}|_{\mathfrak p}\big\} 
}
and
\eq{b2}{\lambda_{\max} := \max \big\{ \lambda :\,\lambda \text{ is an eigenvalue of }\ad_{{g_1}}|_{\mathfrak p}\big\}. }
 {Note that all eigenvalues of the restriction of $\ad_{{g_1}}$ to $\mathfrak p$, including $\lambda_{\min}$ and $\lambda_{\max}$, are positive.}}

{In this section we   derive Theorem \ref{dimension drop 3} from the following crucial {but technical} theorem.} 

\begin{thm}\label{first}
Let $G, \Gamma, X, F = \{g_t: t\ge 0\}$ and $P$ be as in  Theorem \ref{dimension drop 3},  and let $p = \dim P$.
Then there exist {$$r_*, C_1, C_2, a',b', \lambda> 0$$} such that the following holds:\\
For any $0 < c < 1$ there exist $t> 0$ and  a  compact subset  $Q$ of $X$ 
such that: 
\begin{enumerate}
\item
For {all $x \in X$,} and for all $2 \le k \in \N$, the set
{\eq{S1}{{S(k,t,x) :=  \left\{h \in P :{{g_{Nkt}}}hx \notin  Q  \,\,\, \forall N \in \N    \right\}}}}
{satisfies
{\eq{S1 codim}{{\codim {S(k,t,x)}} \ge  { \frac{1}{{\lambda_{\max}} kt }  \log {\frac{1-c}{4c}} }.}}}
\item
For all $2 \le k \in \N$, all $r$ satisfying
{\eq{ineq beta2}{{ {e^{ \frac{a' - kt}{b'}}
} \le r  < \frac{1}{4}\min\big( r_0 \left( \partial_1 Q  \right), {r_*}\big)},}} all $\theta\in {\left[r ,\frac {r_*}{2}\right]}$, all $x \in X$, and for all open subsets $O$ of $X$
we have 
\eq{S2 codim}{ 
{{\codim\left({\left\{h\in P \ssm S(k,t,x) : hx\in \widetilde {E}({F},O) \right\}}\right)
} 
\ge {{ \frac{ \mu \big( {\sigma_{4 \theta}} O  \big)- {{\frac{8C_1}{\theta^{p}}} {{\frac{\sqrt c}{1-c}}}}  - \frac{C_2}{r^{p} }e^{-\lambda kt} }{{\lambda_{\max}} kt}}}}
.}
\end{enumerate}
 \end{thm}

{We now show how the two estimates are put together.}

\begin{proof}[Proof of Theorem \ref{dimension drop 3} assuming Theorem \ref{first}]
{Recall that we are given the constants {$r_*, C_1, C_2, a',b', \lambda > 0
  $ }  such that statements (1), (2) of Theorem \ref{first} hold.
Let $O$ be an open subset of $X$. Define \eq{su1}{\theta_{O}:= \sup \left\{ 0<\theta\le 1: \mu({\sigma_{4 \theta}}O) \ge \frac{1}{2} \mu(O) \right \},
}
then put $ \theta := \min (\theta_O, \frac{r_*}{2})$ and 
\eq{c}{c:= \min \left ({\frac{1}{4 e^{1/2}+1}}, { { \left(\frac{\mu(O)}{128C_1} \cdot \theta^p \right)^2 } }\right). }
Now choose $t$ and $Q$ as in the assumption of Theorem \ref{first}.
Then in view of \equ{c}, statement (1) of Theorem \ref{first} readily implies that for any $2 \le k \in \N$ one has
{\eq{estimate1}{{\codim S(k,t,x)} \ge  \frac{1}{2{\lambda_{\max}} kt  }.  
}}
}
Next,  let 
 \eq{t define}{r:= \frac{1}{{{4} }}\min\big( r_0 \left( \partial_1 Q  \right), {r_*}, \theta_O\big).
 }

{

{
\ignore{, hence $t \ge t_1$. {Furthermore, since $t_1 \ge 4$ it follows from 
\equ{t redefine} that 
\eq{t reredefine}{\ r \le C_{2} e^{-apt}.}}}}
}


Clearly the second inequality in \equ{ineq beta2} is then satisfied.
Now take $2 \le k \in \N$ sufficiently large so that
\eq{klimitx}{{e^{ \frac{a' - kt}{b'}}
} \le r \quad\text{and}\quad \frac{C_2}{r^p}e^{-\lambda kt} \le \frac{\mu(O)}{8};}
{
this will imply the first inequality in \equ{ineq beta2}. Also it is easy to see from  \equ{t define} that 
$\theta\in {\left[r ,\frac{r_*}{2}\right]}$; hence   \equ{S2 codim} holds.

Observe that since $\theta \le \theta_O$,   by  definition of $\theta_O$, 
we have
\eq{tetameasure}{\mu \big({{{\sigma _{{4 {\theta}}}}{ O}}}\big) \le \frac{\mu(O)}{2}.}
Definition of $c$ implies
\eq{climit}{ {{\frac{8C_1}{\theta^{p}}} {{\frac{\sqrt c}{1-c}}} \underset{c<1/2}{<} {\frac{8C_1}{\theta^{p}}} \cdot 2 \sqrt{c} \underset{\equ{c}} \le} \frac{\mu(O)}{8}.}
Hence, by combining \equ{klimitx}, \equ{tetameasure}, and \equ{climit}, we conclude that the numerator in the right hand side of \equ{S2 codim} is not less than $\mu(O)/4$. 
Thus \equ{S2 codim} implies $${\codim\big({\{h\in P \ssm S(k,t,x) : hx\in \widetilde {E}({F},O) \}}\big)} \ge
{\frac{ \mu(O)} {4{\lambda_{\max}} kt }}.
$$
Combining it with  \equ{estimate1}, we obtain
$$
\codim\big({\widetilde{E}(F,O) \cap Px}\big)  \ge  \frac{1}{4{\lambda_{\max}} kt} {\min\big(2 ,
{{\mu(O)} }\big)}  = \frac{ \mu(O)} {4 {\lambda_{\max}} kt } ,
$$
{which is a positive number independent of $x$.}}
This finishes the proof.
\end{proof}

\begin{proof}[Proof of Corollary \ref{dimension drop 2}]
Let $G, \Gamma , X$ and ${F}$ 
be as in Theorem \ref {thmguanshi}  and let $H$ 
be as in \equ{uhs}. Let $\mathfrak{g}$ be the Lie algebra of $G$,
$\mathfrak{g}_\mathbb{C}$ its complexification, and for $\lambda
\in \mathbb{C}$, let $E_ \lambda$ be 
the eigenspace of
$\Ad\, g_1$ corresponding to $\lambda$.
Let $\mathfrak{h}$, $\mathfrak{{h^0}}$, $\mathfrak{{h^ - }}$ be the
subalgebras of $\mathfrak{g}$ with complexifications:
\[{\mathfrak{h}_\mathbb{C}} = \spn({E_\lambda }: \left| \lambda  \right| > 1),\ \mathfrak{h}_\mathbb{C}^0 = \spn({E_\lambda }: \left| \lambda  \right| = 1),\ \mathfrak{h}_\mathbb{C}^ -  = \spn({E_\lambda }:\left| \lambda  \right| < 1).\]
Note that $\mathfrak{h}$ is the Lie algebra of $H$. Moreover, $\mathfrak{h}^ -$ is the Lie algebra of  {the}
{\sl stable horospherical subgroup} defined by
\[{H^ - } := \{ h \in G:{g_t}h{g_{ - t}} \to e\,\,\,as\,\,t \to  + \infty \}. \] 
Since $\Ad\, g_1$ is assumed to be diagonalizable,
$ \mathfrak{g} $ is the direct sum of
$\mathfrak{h}$, $\mathfrak{{h^0}}$ and $\mathfrak{{h^ - }}$. Hence, if we denote the group ${H^ - }{H^0}$ by $\widetilde H$, $G$ is
{locally (at a neighborhood of identity) a} direct product of $H$ {and} $\widetilde {H }$. 

{Now let $O$ be a non-empty open subset of $X$, and fix ${0 < \rho < 1}$ such that the following properties are satisfied:}
\begin{enumerate}
    \item []
{the multiplication map } $\widetilde H \times H\to G$ { 
is one to one on} $B^{{\widetilde H}}(\rho) \times {B^H}(\rho),$
\item []
\eq{conjugate implied}
{g_tB^{\widetilde H}(\rho)g_{-t} \subset B^{\widetilde H}(2\rho) \text{ for any } t\ge 0, } 
\item []
and
\eq{nonemp}{\sigma_{2 \rho}O \neq   \varnothing.}
\end{enumerate}
\smallskip
Note that \equ{conjugate implied} can be {satisfied}  since $F$ is $\Ad$-diagonalizable and the restriction of the map $g \to g_tgg_{-t}$,
$t \ge 0$, to $\widetilde H$
is non-expanding. Also \equ{nonemp} can be {achieved, since in view of \equ{inner-core} 
$\sigma_r O$ is   non-empty}  when $r>0$ is sufficiently small.

 Now in view of \equ{nonemp}, we can apply Theorem \ref{dimension drop 3} with $O$ replaced with $\sigma_{2\rho}O$ and conclude that there exists $\vre>0$ such that
\eq{sliceH}{\dim \left(\big\{h\in H: hx\in \widetilde E({F},\sigma_{2 \rho}O)\big\} \right)= \dim H - \vre  < \dim H                  .}
Choose $s>0$ such that ${B
}(s)$ is contained in the product  $B^{{\widetilde H}}(\rho){B^H}(\rho)$, and for  $x\in X$ denote 
\begin{equation*}{{E_{x}}:=}\, \big\{ g \in {B
}(s):gx \in \widetilde {E}({F },O)\big\}.
\end{equation*}
%
In view of the countable stability of Hausdorff dimension, in order to prove the corollary it suffices to prove that for any  $x \in X$,
\begin{equation*}\dim E_{x}  \le  \dim X - \vre,\end{equation*}
where $\vre$ is as in \equ{sliceH}; note that $\widetilde{E}({F }, O)$ can be covered by countably many sets $\{gx : g \in E_{x}\}$, with the maps 
{$\pi_x: E_{x} \to X$}
being Lipschitz and at most finite-to-one. 
Since every ${g\in B(s)}
$ can be written as $g = h'h$, where
$h' \in {B^{\widetilde H}}(\rho)$ and $h \in {B^{H}}(\rho)$, 
for any $y \in X$ we can write
\begin{equation*}\begin{aligned}
{\dist}({g_t}gx,y) &\le {\dist}({g_t}h'hx,{g_t}hx) + {\dist}({g_t}hx,y)\\ &={\dist}\big(g_th'g_{-t}{g_t}hx,{g_t}hx\big)+ {\dist}({g_t}hx,y).\end{aligned}\end{equation*}
Hence  in view of \equ{conjugate implied},
$g \in {E_{x}}$ implies that {$h{x}$ belongs to $E({F },{\sigma
_{{2 \rho}}}U)$}, and by using
Wegmann's Product Theorem \cite{Weg} we have: 
$$
{\begin{aligned}
\dim {E_{x}} 
&\le {\dim} \left(\{ h \in {B^H}(\rho):hx \in E({F },{\sigma
_{2 \rho}}O)\} \times {B^{\widetilde H}(\rho)}
\right)\\
& \le {\dim} \big(\{ h \in {B^H}(\rho):hx \in \widetilde{E}({F },{\sigma
_{2 \rho}}O)\}\big) +  \dim  {\widetilde H } \\
& \underset{\equ{sliceH}}\le \dim H - \vre + \dim \widetilde {H}  = \dim X - \vre .
\end{aligned}
}$$
This ends the proof of the corollary.
\end{proof}


\section{Tessellations and  
{Bowen boxes}
}\label{tess}
{Let} $P$ be a {connected} subgroup of $H$ normalized by $F$. Following \cite{KM1}, 
say that an open subset $V$ of $P$ is a {\sl tessellation domain} 
relative to a countable subset $\Lambda$ of $P$ if 
\begin{itemize}
\item
$\nu (\partial V) = 0$;
\item
$V \gamma_1 \cap V \gamma_2 = \varnothing$ for different $\gamma_1,\gamma_2 \in \Lambda$;
\item
$P = \bigcup\limits_{\gamma  \in \Lambda } {\overline V \gamma }.$
\end{itemize} 

{Note that $P$ is a connected   simply connected nilpotent Lie group. Denote \linebreak $\mathfrak p := \Lie(P)$  and $p:= \dim P$.
As shown in  \cite[Proposition 3.3]{KM1}, one can choose a basis of $\mathfrak p$ such that for any $r  > 0$, 
{$\exp \left(rI_{\mathfrak p}\right)$, where $I_{\mathfrak p} \subset \mathfrak p $ is the cube centered at $0$ with side length $1$ with respect to that basis, is a tessellation domain. Let us denote
\eq{defvr}{V_r :=\exp \left({{\frac{r}{4 \sqrt{p}}}} 
{I_{\mathfrak p}}\right)} 
and choose a countable $\Lambda_r\subset P$   such that $V_r$  is a tessellation domain 
relative  to $\Lambda_r$.} 

Take $0<r_*<1/4$ such that the exponential map from $\mathfrak p$ 
  to $P$ is $2$-bi-Lipschitz on $
 B^{\mathfrak p}(r_*)$. 
The latter implies that
{\eq{Bowen inc}{{B^P}\Big(\frac{r}{{16  \sqrt{p} }}\Big) \subset {V_r} \subset {B^P}\left( \frac r {4}\right)\quad  \text{for any }0<r   \le r_*.}
Also, the measure $\nu$ and the pushforward of the Lebesgue measure $\Leb$ on $\mathfrak p$ are absolutely continuous with respect to each other with locally bounded Radon--Nikodym derivative. This implies that there exists $0< c_1< c_2$ such that 
\eq{lb1}{
c_1\Leb(A) \le \nu\big(\exp(A)\big)\le c_2\Leb(A)\quad \forall\, \text{measurable } A\subset  B^{\mathfrak p}(1).
}

In what follows we will be taking $\theta \ge r$ and approximating ${V_\theta}$ by the union of $\Lambda_r$-translates of ${V_r}$. The following estimate will be helpful:

\begin{lem}\label{coveringtheta} 
 For any  $0<r  \le \theta \le r_*/2$ 
$$
\#\{\gamma\in \Lambda_r: {V_r} \gamma \cap {V_\theta} \neq \varnothing \}\le \frac{c_2}{c_1}\left(\frac \theta r + 8\sqrt p \right)^p.
$$ 
\end{lem}

\begin{proof} \ignore{For the lower bound, observe that $\#\{ \gamma \in {\Lambda_r}: V_r \gamma \subset {V_\theta}              \}  \ge 1$ because of the convention $e\in\Lambda_r$. Then note that for any $\gamma$ we have $\diam(V_r \gamma) < r/2$ in view of \equ{Bowen inc} and the right-invariance of the metric.  So if $V_r \gamma$ contains a point of $\sigma_{r/2}{V_\theta}$, then it must lie in the interior of ${V_\theta}$. Therefore,
$$ \begin{aligned}
\#\{ \gamma \in {\Lambda_r}: V_r \gamma \subset V_\theta              \} 
& \ge \frac{\nu (\sigma_{r/2} V_\theta)}{\nu(V_r)} \underset{\equ{lb1}}{\ge} \frac{c_1}{c_2} \cdot   \frac{ \Leb \left( \sigma_r \big(\frac{\theta}{4 \sqrt{p}} {I_{\mathfrak p}} \big) \right)}{ \Leb \left(\frac{r}{4 \sqrt{p}} {I_{\mathfrak p}} \right)}   \\
(\text{if } \frac{\theta}{4 \sqrt{p}} - 2r  \ge 0)\qquad&=\frac{c_1}{c_2} \cdot \frac{\left( \frac{\theta}{4 \sqrt{p}} - 2r   \right)^p}{\left(\frac{r}{4 \sqrt{p}} \right)^p} 
= \frac{c_1}{c_2}   \Big(\frac \theta r - 8\sqrt p\Big)^p, 
\end{aligned}$$
which gives us the lower bound that we were looking for. Note that in the second inequality above we also used the bi-Lipschitz property of $\exp$.  

For the upper bound, n}
Note that if ${V_r} \gamma$ intersects ${V_\theta}$, then in view of \equ{Bowen inc} 
we must have ${V_r} \gamma \subset \partial_{r/2}{V_\theta}$. Hence,
$$
\begin{aligned}
\#\{\gamma\in \Lambda_r: {V_r} \gamma \cap {V_\theta} \neq \varnothing \} 
&\le \frac{\nu \left(\partial_{r/2}{V_\theta} \right)}{\nu \left( {V_r}\right)}\underset{\equ{lb1}}\le \frac{c_2}{c_1} \cdot   \frac{ \Leb \left( \partial_r \big({\frac{\theta}{4 \sqrt{p}} {I_{\mathfrak p}}} \big) \right)}{ \Leb \left({\frac{r}{4 \sqrt{p}} {I_{\mathfrak p}}} \right)} \\
& 
= 
\frac{c_2}{c_1} \cdot \frac{\left( \frac{\theta}{4 \sqrt{p}} + 2r   \right)^p}{\left(\frac{r}{4 \sqrt{p}}  \right)^p} 
 = \frac{c_2}{c_1}\left(\frac \theta r + 8\sqrt p \right)^p,
\end{aligned}$$
where in the second inequality above we were able to use the bi-Lipschitz property of exp 
since $$ \partial_r \big({\frac{\theta}{4 \sqrt{p}} {I_{\mathfrak p}} }\big)\subset {B^{\mathfrak p} \left({\frac{\theta}{{8}}  + r   }\right)} \subset B^{\mathfrak p}(r_*).$$
This finishes the proof.
\end{proof}
\smallskip

{Recall that all eigenvalues of the restriction of $\ad_{{g_1}}$ to $\mathfrak p$ are positive. 
{Using the  bi-Lipschitz property of $\exp$, one can conclude that}}
{\eq{diam}{
\begin{aligned}\diam(g_{-t}{V_r}g_t)&\le  2\cdot \diam\left(\exp {\Big({{{\frac{re^{-\lambda_{\min} t}}{4 \sqrt{p}}}} 
{I_{\mathfrak p}}}\Big)} \right) \\ & \le \frac {re^{-\lambda_{\min} t}}2   \quad\text{for any $0<r \le r_*$ and any }t\ge 0
,
\end{aligned}}}
{where $\lambda_{\min}$ is as in \equ{b1}.}
{Also let 
$\delta := \Tr\ad_{{g_1}}|_{\mathfrak p}$;
clearly one then has} \eq{delta}{{\nu(g_{-t}A g_t) = e^{-\delta t}\nu(A)\text{ for any measurable }A\subset P.}}

{Let us now define} a {\sl Bowen $(t,r)$-box} in $P$ 
{to be} a set of the form $g_{-t}\overline{V_r} \gamma g_t$  for some $\gamma \in P$ {and $t \ge 0$}. 
The following lemma, analogous to {\cite[Proposition 3.4]{KM1} and} \cite[Lemma 6.1]{KMi1}, gives   an upper bound for 
the number of $\gamma  \in \Lambda_r$ such that the Bowen box {$g_{-t}\overline{V_r} \gamma g_t$ 
has non-empty intersection with $\overline{V_r}$:
\begin{lem} 
\label{covering} For any  $0<r  \le r_*/2$ {and 
\eq{bigt}{t\ge\frac{\log (8\sqrt{p})}{\lambda_{\min}} ,}
one has}\[\# \{ \gamma  \in \Lambda_r :{g_{ - t}}{\overline{V_r}\gamma}{g_t}  \cap \overline{V_r} \ne \varnothing \}  \le 
e^{\delta t} \left(1 + {C_0}e^{-\lambda_{\min} t}\right),\]
where \eq{defd0}{{C_0} :=  
 {\frac{ 2^{p+3} {p}^{3/2} c_2}{c_1}}.}
\end{lem} 

\begin{proof}
Let $0<r \le r_*/2$.
One has:
\[
\begin{aligned}
\# 
&  \{ \gamma  \in \Lambda_r :{g_{ - t}}{\overline{V_r}\gamma}{g_t}  \cap \overline{V_r} \ne \varnothing \} \\
= \#  &\{ \gamma  \in \Lambda_r :{g_{ - t}}{{V_r}\gamma}{g_t}  \subset {V_r}\}  + \# \{ \gamma  \in \Lambda_r :{g_{ - t}}{\overline{V_r}\gamma}{g_t}  \cap \partial {V_r} \ne \varnothing \}. \end{aligned}\]
 Since $V_r$ is a tessellation domain of $P$ relative to $\Lambda_r$, the first term in the above sum is not greater than $\frac{\nu({V_r})}{\nu(g_{-t}{V_r}g_t)} = e^{\delta t}$, {while, in view of 
 \equ{diam}, the second term is not greater than
\eq{2term}{{\frac{{\nu \Big(
\partial_{{\frac {re^{-\lambda_{\min} t}}2}}(\partial {{V_r}})
\Big)
}}{\nu ({g_{ - t}} {V_r}{g_t})}}
\underset{\equ{lb1},\,\equ{delta}}\le c_2e^{\delta t}\ \frac{{\Leb \left(
\partial_{{r    e^{-\lambda_{\min} t}}}\big(\partial ({{\frac{r}{{4} \sqrt{{p}}} {I_{\mathfrak p}}}})\big)
\right)}}
{\nu ( {V_r})}
.}
(Here we used the fact that $${\partial_{r    e^{-\lambda_{\min} t}}\left( {\frac{r}{{4} \sqrt{p}} {I_{\mathfrak p}}} \right)\subset {B^{\mathfrak p}\left(\frac{r}{{8}}  + r    e^{-\lambda_{\min} t}\right)} \subset {B^{\mathfrak p}\left(\frac{9r}8\right)} \subset B^{\mathfrak p}\left(\frac{9r_*}{16}\right)}\subset B^{\mathfrak p}(1),$$
hence we can use the $2$-bi-Lipschitz property of $\exp$ to conclude that $$\exp \left(
\partial_{{r    e^{-\lambda_{\min} t}}}\big(\partial ({{\tfrac{r}{{4} \sqrt{{p}}} {I_{\mathfrak p}}}})\big)
\right)\supset\partial_{{\frac {re^{-\lambda_{\min} t}}2}}(\partial {V_r}),$$ and the estimate \equ{lb1} is applicable.) 
It is easy to see that the numerator in the right hand side of \equ{2term} is not greater that 
{$$
\begin{aligned}
&\qquad  \left(  \frac{r}{4 \sqrt{p}}  + 2r    e^{-\lambda_{\min} t}   \right)^{p}  -   \left(  \frac{r}{2 \sqrt{p}}  - 2r    e^{-\lambda_{\min} t}   \right)^{p} \\
&  \underset{\text{(Mean Value Theorem)}} \le 4r    e^{-\lambda_{\min} t} p  \left(  \frac{r}{4 \sqrt{p}}  + 2r    e^{-\lambda_{\min} t}   \right)^{p-1} \\
& \underset{\equ{bigt} } \le  4pr    e^{-\lambda_{\min} t}   \left(  \frac{r}{2 \sqrt{p}}    \right)^{p-1} 
= 
2^{p+3} {p}^{3/2} \left(  \frac{r}{4 \sqrt{p}}     \right)^p{e^{-\lambda_{\min} t}}
  \underset{\equ{defvr},\,\equ{lb1} }\le \frac{ 2^{p+3} {p}^{3/2} }{c_1}\nu({V_r})e^{-\lambda_{\min} t}
,\qquad\end{aligned}$$} which finishes the proof.}
 \end{proof}
  
{We conclude the section with a  lemma, {which is a slight modification of \cite [Lemma 6.4]{KMi1}}, to be used at the last stage of the proof for switching from coverings by Bowen boxes to coverings by balls.
\begin{lem}
\label{coveringballs} For any {$t>0 $} and any $0<r \le r_*$,  any Bowen $(t,r)$-box in $P$ can be covered with at most  {$ e^{(p \lambda_{\max}- \delta)t}$}
balls 
of radius $r e^{- \lambda_{\max}t}$, {where $\lambda_{\max}$ is as in \equ{b2}}.
\end{lem}
\begin{proof}
Using the  $2$-bi-Lipschitz property of $\exp$ again, one can cover $g_{-t}\overline{V_r}g_t$ by at most as many  balls of radius $r e^{- \lambda_{\max}t}$,  as the number of  translates of   $ \frac{re^{- \lambda_{\max}t}} {\sqrt{p}}I_{\mathfrak p}$ needed to cover $\Ad(g_{-t})\left( \frac{r} {4\sqrt{p}}\overline{I_{\mathfrak p}}\right)$. The latter can be written as the direct product of intervals $I_1,\dots,I_p$, where $\min_i\Leb(I_i) =  \frac{re^{- \lambda_{\max}t}} {4\sqrt{p}}$. Clearly each $I_i$ can be covered by the union of intervals of length $ \frac{re^{- \lambda_{\max}t}} {\sqrt{p}}$ whose total measure is at most $4\Leb(I_i)$. Hence $\Ad(g_{-t})\left( \frac{r} {4\sqrt{p}}\overline{I_{\mathfrak p}}\right)$ can be covered by at most $$
\frac{4^p \Leb\left(\Ad(g_{-t})\big( \frac{r} {4\sqrt{p}}{I_{\mathfrak p}}\big)\right)}{\Leb\left( \frac{re^{- \lambda_{\max}t}} {\sqrt{p}}{I_{\mathfrak p}}\right)} = \frac{4^p e^{- \delta t}\Big( \frac{r} {4\sqrt{p}}\Big)^p}{\left( \frac{re^{- \lambda_{\max}t}} {\sqrt{p}}\right)^p} =   e^{(p \lambda_{\max}- \delta)t}$$
 translates of   $ \frac{re^{- \lambda_{\max}t}} {\sqrt{p}}I_{\mathfrak p}$, which finished   the proof of the lemma.
\end{proof}
}


\section{{Property (EEP) and a measure estimate
}}\label{measest}

\ignore{Applying a linear time change to the flow $g_t$, without {loss of}  generality we can assume that 
Define
\gr
{\eq{eigen value}
{\lambda' = 1.}}}
 
 {Our 
 goal in this section is to use property (EEP) of $P$ to find a lower bound for the measure of the sets of the type \eq{ax}{\{h \in
{V_r} :{g_{t}}hx \in O\},} where $x\in X$, $O$ is a  subset of $X$, $r>0$ is small enough, and $t > 0$ is large enough.}  {This step is similar to \cite[Theorem 4.1]{KMi1}, where balls in $P$ were used in place of tessellation domains $V_r$. For our new proof the use of tessellations is crucial; to make the paper self-contained we present a complete argument.}

{We start with the definition of Sobolev spaces.
Let $G$ be a 
Lie
group  
and  $\Gamma$ 
a discrete subgroup of $G$ such that  $X=\ggm$ admits a $G$-invariant measure $\mu_X$.  
{Fix a basis $\{Y_1,\dots,Y_N\}$ 
for the Lie algebra
$\mathfrak g$
of
$G$, and, given 
$h \in C^\infty(X)$, $k\in\N$  and $\ell\in{\Z_+}$, define the ``{\sl $L^{{k}}$, order $\ell$" Sobolev norm} $\|h \|_{\ell{,k}}$ of $h $ by 
 $$
 \|h \|_{\ell{,k}} \df \sum_{|\alpha| \le \ell}\|D^\alpha h \|_{{k}},
 $$ where $\|\cdot\|_{{k}}$  stands for the $L^k$ norm,
 $\alpha = (\alpha_1,\dots,\alpha_N)$ is a multiindex, $|\alpha| = \sum_{i=1}^N\alpha_i$, and $D^\alpha$ is a differential operator of order $|\alpha|$ which is a monomial in  $Y_1,\dots, Y_N$, namely $D^\alpha = Y_1^{\alpha_1}\cdots Y_N^{\alpha_N}$.
This definition depends on the basis,
however, a change of basis
would only  distort
$ \|\cdot \|_{\ell{,k}}$
by a bounded factor. 
We will also use the operators $D^\alpha$ to define $C^\ell$ norms of smooth 
functions $f$ on $X$: 
$$
\|f\|_{C^\ell}  := \sup_{x\in X, \ |\alpha|\le \ell}|D^\alpha f(x)|.
$$}}

{The next} 
lemmas provide a way to approximate subsets of $X$ and $P$ respectively by smooth functions. 
{We start with a basic lemma constructing  test functions supported inside small neighborhoods of identity in $G$. It} is an immediate corollary of \cite[Lemma 2.6]{K}, see also \cite[Lemma 2.4.7(b)]{KM1}.

\begin{lem}\label{KMlem}
For each {$\ell\in \Z_+$} there exists   $ {M_{G,\ell}}\ge 1$ 
with the following property: for any
$0 < {{\vre}} < 1$ there exists a nonnegative smooth function  $\varphi_{{\vre}}$ on $G$ such that
{
\begin{enumerate}
\item the support of $\varphi_{{\vre}}$ is inside $B(\vre)$;
\item $\|\varphi_{{\vre}}\|_1
= 1$; 
\item
$\| \varphi_{{\vre}} \|_{\ell{,1}} \le {M_{G,\ell}} \cdot {{\vre}}^{-\ell}$.
\end{enumerate}}
\end{lem}

The {next lemma} is a slightly easier version of \cite[Lemma 5.2]{KMi1}; we provide the proof for the sake of completeness.

\begin{lem}
\label{estimate} 
{For any $\ell \in \Z_+$ there exist a constant
$M_\ell> 0$ (depending only on   $\ell$ {and $G$})} such that for any nonempty open subset $O$ of $X$ and 
any $0<\varepsilon <1$ 
one can find {a nonnegative function $\psi_\varepsilon \in C
^\infty (X)$} such that 
\begin{enumerate}
\item $ {1_{\sigma_\varepsilon O}}\le {\psi _\varepsilon } \le {1_O}$;
 \item $\max\big(
{\left\| {{\psi _\varepsilon }} \right\|_{\ell{,2}} } , {\left\| {{\psi _\varepsilon }} \right\|_{C^\ell} }\big)\le 
{M_\ell }{\varepsilon ^{ - \ell}}$.
\end{enumerate}
\end{lem}

{\begin{proof}
Let $O$ be a nonempty open subset of $X$, and let {$0< \varepsilon <1$}.
{Now 
take ${\psi _\varepsilon } =
\varphi_{{\vre/2}}*{1_{\sigma _{{\varepsilon/2}}O }}$, where $\varphi_{{\vre/2}}$ is as in {Lemma \ref{KMlem}}. 
It follows from the definition of $\psi_\varepsilon$ and the normalization $\|\varphi_{{\vre}}\|_1
= 1$ that $\psi_\varepsilon(x) \le 1$ for all $x$. Also,  {since} $\varphi_{{\vre/2}}$ is supported on 
$B (\varepsilon/2)$, the support of the function
$\psi_\varepsilon$ is contained in $\partial_{\varepsilon/2}\sigma _{{\varepsilon/2}}O \subset
O$,} which implies ${\psi _\varepsilon } \le {1_O}$.
Furthermore, if $x\in \sigma_\varepsilon O$ and $g\in B(\varepsilon/2)$, then $gx\in \partial_{\varepsilon/2}\sigma_\varepsilon O \subset \sigma _{{\varepsilon/2}}O$, i.e.\ $1_{\sigma _{{\varepsilon/2}}O }(gx) = 1$. Therefore 
$$
{\psi _\varepsilon }(x) = \int_G
\varphi_{{\vre/2}}(g){1_{\sigma _{{\varepsilon/2}}O }}(gx)\,d {\mu_G} = \int_G
\varphi_{{\vre/2}}(g)\,{d \mu_G} = 1.$$
So property $(1)$ holds. 

{Let $\alpha= (\alpha_1,\dots,\alpha_N) $ {be} such that $\left| \alpha  \right| \le \ell$. For any $x \in X$ we have 
}
{\begin{equation*}{\begin{aligned}
 \left| {{D^\alpha }{\psi _\varepsilon }(x)} \right| &
 = \left| {{D^\alpha }( \varphi_{\vre/2}  * 1_{\sigma _{{\varepsilon/2}}O })(x)} \right| =  \left| {{D^\alpha }\varphi_{\vre/2}  * 1_{\sigma _{{\varepsilon/2}}O }(x)} \right| 
 \\
  &
   \le {\left\| D^ \alpha{{\varphi_{{\vre/2}}}} \right\|}_1  \le {\left\| {{\varphi_{{\vre/2}}}} \right\|}_{\ell,1} \le M_{G,\ell}  {(\tfrac\varepsilon{2}) ^{ - \ell}},
\end{aligned} }
\end{equation*}
where $M_{G,\ell}$ is as in Lemma \ref{KMlem}.
{Likewise},  
by Young's inequality, 
$${{\left\| D^ \alpha{{\psi _\varepsilon }} \right\|}_2 \le \|{D^\alpha }\varphi_{\vre/2}  * 1_{\sigma _{{\varepsilon/2}}O } \|_2 \le {\left\| D^ \alpha{{\varphi_{{\vre/2}}}} \right\|}_1 \cdot {\left\| 1_{\sigma _{{\varepsilon/2}}O }\right\|}_2 \le  {\left\| D^ \alpha{{\varphi_{{\vre/2}}}} \right\|}_1 \le M_{G,\ell}  {(\tfrac\varepsilon{{2}}) ^{ - \ell}},}
$$
which implies (2) 
with $M_\ell =  {2}^\ell M_{G,\ell}$.}
\end{proof}

The next lemma is a modification of \cite[Lemma 5.3]{KMi1} where we replace balls of radius $r$ in $P$ with $V_r$; we omit the proof.

\begin{lem}
\label{ball estimate}
{For any $\ell \in \Z_+$ there exist constants
$M'_\ell\ge 1$  (depending only on $\ell$ and $P$)}
such that the following holds:
for any $0< \varepsilon,r  \le r_*/2$, 
 there exist functions ${f_\varepsilon }:P \to [0,1]$ such
that
\begin{enumerate}
\item ${f_\varepsilon } = 1\,\,on\,\,\,V_r$;
\item ${f_\varepsilon } = 0\,\,on\,\,\,{\big( V_{r+\vre}\big)^c}$;
{\item $\max\big({\left\| {{f_\varepsilon }} \right\|_{\ell{,2}}} \,, {\left\| {{f_\varepsilon }} \right\|_{C^\ell}}\big) \le {M'_{\ell}}{\varepsilon ^{ - 
\ell}}$. }
\end{enumerate}
\end{lem} 
 
\ignore{For any {$x\in X$ denote by $\pi_x$ the map  $G\to X$ given by $\pi_x(g) = gx$, and by  $r_0(x)$ the \textsl{injectivity radius} of $x$, defined as $$
\sup\{r > 0: 
\pi_x\text{ is injective on }B(r)\}.$$
If $K\subset X$ is bounded, let us denote by $r_0(K)$ the \textsl{injectivity radius} of $K$: $$
r_0(K) := \inf_{x\in K}r_0(x) = \sup\{r > 0: 
\pi_x\text{ is injective on }B(r)\  \ \forall\,x\in K\}.$$}}

{Here is the main result of the section, {which is a modified and improved version of \cite[Proposition 5.1]{KMi1}}.}
\begin{prop} \label{exponential mixing} {Let  $P$ be a subgroup of $G$ with property {\rm (EEP)} {with respect to the flow $(X,F)$}. Then 
for any open $O \subset X$,  
any $x\in X$,  
any 
\eq{rlimit}{0<r<\frac{1}{2} \min \big(r_0(x), r_*\big),} 
and any $t$ satisfying 
\eq{t estimate} {t \ge a'+ b \log \frac{1}{r_0(x)} }
one has}
$${ {\mathop  \nu \big(
{\{h \in
{V_r} :{g_{t}}hx \in O\}}\big)\ge \nu\left({V_r}\right)\mu ({\sigma _{e^{ - \lambda 't}}}O)  - {e^{ - {\lambda 't}}}.}}$$
Here
\eq{exponent}{\lambda' := \frac{\lambda}{4 \ell +2} }
and
\eq{a'}{a':=\max \left(a,\frac{1}{\lambda '} \log \left({ M_{\ell} {M'_{\ell}} {{E}}} + {p} {c_2} \right),\frac{\log \frac{2}{r_*}}{2 \lambda '} \right),}
where   
{$\ell,\lambda,a,b$} are as in Definition \ref{subgroup}, {$E$ is an implicit constant from \equ{eep}}, $c_2$ is as in \equ{lb1}, {and} $M_\ell , M'_\ell$ are as in Lemmas   \ref{estimate} {and \ref{ball estimate}}. 

\end{prop}

\begin{proof}[Proof of Proposition \ref{exponential mixing}] 
{Let  $O \subset X$ be an open subset of $X$, and take $x \in X$ and 
$r$ as in \equ{rlimit}.}
{Now} 
set $f = {1_{{V_r}}}$, 
 {take  $t$  as in \equ{t estimate} and put $\vre:=e^{-2 \lambda' t}.$ 
}
{Note that \equ{t estimate} and  
\equ{a'}
give  
\eq{llimit}{\vre \le \frac{r_*}{2} .}
{Now} let ${\psi
_\varepsilon}$ and ${f _\varepsilon}$ be the functions constructed in Lemmas \ref{estimate} 
and 
 \ref{ball
estimate} respectively.
Then we have 
\eq{prime}{
\begin{aligned}
\max (\|\psi_{{\varepsilon}}\|_{C^1},{\left\| \psi_{{\varepsilon}}  \right\|_{\ell{,2}} }) \cdot {\left\| f_{{\varepsilon}} \right\|_{{C^\ell }}} \cdot {e^{ - \lambda t}} 
& \le \max (\|\psi_{{\varepsilon}}\|_{C^\ell},{\left\| \psi_{{\varepsilon}}  \right\|_{\ell{,2}} }) \cdot {\left\| f _{{\varepsilon}}\right\|_{{C^\ell }}} \cdot {e^{ - \lambda t}} \\ 
& \le  
M_{\ell} {\varepsilon ^{ - \ell }}{M'_{\ell}}{\varepsilon ^{ - \ell }} e^{-\lambda t} \\
&  
=  
M_{\ell} {M'_{\ell}} {e^{4 \ell \lambda 't-\lambda t}} 
\underset{\equ{exponent}} \le   
 M_{\ell} {M'_{\ell}} {e^{-2 \lambda ' t}} .
\end{aligned}}
{Furthermore}, by \equ{Bowen inc} and {\equ{llimit}},
\eq{supinc}{\supp f_{\vre} \subset {V_{r+{e^{-2 \lambda ' t }}}} \subset V_{r + \frac{r_*}{2} } \subset V_{r_*} \subset B^P(1).}
Also, in view of 
\equ{a'} {we have $a' \ge a$; hence, {inequality} \equ{conditionont} is satisfied for any $x  ,t$ satisfying \equ{t estimate}.}
Hence the estimate \equ{eep} can be applied to ${\psi
_\varepsilon}$, ${f _\varepsilon}$, $x$ and $t$, and, in view of \equ{prime}, yields}
$${\begin{aligned}
\int_P {{f_\varepsilon }({h})} \psi_\varepsilon ({g_{t}}{h}x)\,d\nu ({h}) 
&\ge \int_P {{f_\varepsilon }\,d} \nu \int_X {{\psi _\varepsilon }{\mkern 1mu} \,d\mu }  - {
M_{\ell} {M'_{\ell}} {{E}}}{e^{ -2 \lambda ' t}}.
 \end{aligned}}$$
{Thus} we have:
$${\begin{aligned}
& \nu \left({\{h \in
{V_r} :{g_{t}}hx \in O\}} \right) 
= \int_P {f({h}){1_{ O}}({g_t}} {h}x)\,d\nu ({h})\\ 
& \ge \int_P {f({h})\psi_\varepsilon ({g_{t}}} {h}x)d\nu ({h})  \ge \int_P {{f_\varepsilon }({h})\psi_\varepsilon ({g_{t}}} {h}x)\,d\nu ({h}) - \int_{P} {|{f_\varepsilon } - f|\,d\nu}  \\
&  \ge \int_P {{f_\varepsilon }({h})\psi_\varepsilon ({g_{t}}} {h}x)\,d\nu ({h}) - \nu\left(V_{r + {e^{ - 2 \lambda ' t}}}\ssm {V_r}\right) 
.\end{aligned}
}$$
{Since by \equ{rlimit} and \equ{llimit} we have $r + e^{-2  \lambda' t} \le r_*$,} {it follows that $\frac{r + {e^{ - 2 \lambda ' t}}}{{4} \sqrt{p}}  {I_{\mathfrak p}} \subset B^{\mathfrak p}(1)$. So in view of \equ{lb1}}
\begin{equation*}\begin{aligned}
& \nu\left(V_{r + {e^{ - 2 \lambda ' t}}}\ssm {V_r}\right)  \le c_2 \Leb \left(\frac{r + {e^{ - 2 \lambda ' t}}}{{4} \sqrt{p}}  {I_{\mathfrak p}} \ssm {\frac{r}{{4} \sqrt{p}} {I_{\mathfrak p}}} \right)  \\ 
& { \underset{\text{(Mean Value Theorem)}}{\le} c_2 \left(\frac{1}{\sqrt{4p}}\right)^{p}e^{ - 2 \lambda ' t}p\big(r + e^{ - 2 \lambda ' t} \big)^{{p-1}} \le c_2p e^{2 \lambda't}.}
\end{aligned}\end{equation*}
{Combining the above computations, we obtain}
\begin{align*}
\nu \left({\{h \in
{V_r} :{g_{t}}hx \in O\}}\right)
&
\ge \int_P {{f_\varepsilon }({h})\psi_\varepsilon ({g_{t}}} {h}x)\,d\nu ({h}) - {c_2}{p}{e^{ -2 \lambda ' t}} \\
&\ge \int_P {f_\varepsilon }\,d\nu\int_X {\psi_\varepsilon \,d\mu  - { M_{\ell} {M'_{\ell}} {{E}}} {e^{ - 2 \lambda ' t}}}  -   {c_2}{p}{e^{ -2 \lambda ' t}} \\
&\ge  \nu \left( {V_r} \right)  {{\mu ({\sigma _{\vre}}O)}} - ( { M_{\ell} {M'_{\ell}} {{E}}} + {c_2}{p}){e^{ - 2 \lambda ' t}} \\
& = \nu \left( {V_r} \right)\mu ({\sigma _{e^{- 2 \lambda ' t}}}O) -( { M_{\ell} {M'_{\ell}} {{E}}} + {c_2}{p}) {e^{ - 2 \lambda ' t}} \\
& \underset{ \equ{a'}}{\ge}  \nu \left( {V_r} \right)\mu ({\sigma _{e^{-  \lambda ' t}}}O) - {e^{ -  \lambda ' t}}.
\end{align*}
\end{proof}

\section{{Coverings by  Bowen boxes
}}\label{boxes}

{For $x \in X$, $t>0, {N \in \N}$ and {a} subset $S$ of $X$ let us define
\eq{escape 1}{{{A}^N_x(t,r,S)  :=  \big\{ h \in
\overline{V_r} :{g_{{\ell} t}}hx \in S \,\,\,  {\forall\, \ell \in \{1,\dots,N  \}\mkern 1mu}   \big\}}.}
Clearly the set \equ{ax} studied in the previous section 
{has the same measure as} ${A}^1_x(t,r,O)$.}
{Our goal in this section will be to inductively use Proposition \ref{exponential mixing} to find an effective covering result for the set {${A}^N_x(t,r,O^c)$}.
We start with the following theorem, {which is a modified and improved version of \cite[Proposition 5.1]{KMi1}}:

\begin{thm}\label{main cov} Let ${F}$ be a one-parameter {$\Ad$-}diagonalizable
{sub}semigroup of $G,$ and $P$ a subgroup of $G$ that has  {property} {\rm (EEP)} with respect to the flow $(X,F)$. Then there exist positive constants {$a',b'  \ge {\frac{\log ({8}\sqrt{p})}{\lambda_{\min}}} , C_2, \lambda$} 
such that 
for any {open $O\subset X$}, 
any
\eq{r estimate 2}{0<r<\frac{1}{2} \min \Big(r_0\big(\partial_{1/2} (O^c)\big), r_*\Big),}
any $x\in \partial_{
r}{(O^c)},$ any $t$ satisfying
{\eq{t estimate 2}{t \ge a'+b' \log \frac{1}{r},}} 
and any  {$N\in\N$}, 
the set ${A}^N_x(t,r,O^c)$ can be covered with at most
$$
{e^{\delta Nt}} \left( {1-\mu ({\sigma _{r}}O) } +\frac{C_2}{r^{p}} e^{-\lambda t} \right)^N$$
Bowen $(Nt,r)$-boxes in $P$. 
\end{thm}}

{We remark that the above theorem, as well as its corollary proved later in this section, is applicable only to the situations when  the complement of $O$ is compact: indeed,  otherwise $r_0\big(\partial_{1/2} (O^c)\big) = 0$ and \equ{r estimate 2} is never satisfied.}

\smallskip

Before we prove the theorem, we need the following lemma:
{\begin{lem}
\label{sub count} For any
{$x \in X$, any $O \subset X$, any $0<r  \le r_*$  and  any $t>0$ we have} 
\eq{counting}{  \#\{\gamma \in {\Lambda_r} : \overline{V_r} \gamma g_t x \subset O \}
 \ge \frac{\nu\left({A}^1_x(t,r,{\sigma _{r/2}}{O})\right)}{\nu({g_{ - t}}{V_r} {g_t})}.           }
\end{lem}
\begin{proof}
For any $\gamma \in {P}$ and any ${ h}_1,{ h}_2 \in \overline{V_r}$ 
we have:
{
\eq{bd}
{\begin{aligned}
{\dist}\big({h}_1\gamma g_t  x,{ h}_2\gamma g_t x\big) 
 \le {\dist}({ h}_1,{ h}_2) \le \diam(\overline{V_r}) \le r/2.
\end{aligned}}
}
{Hence, if $$ {A}^1_x(t,r,{\sigma _{r/2}}{O}) \cap g_{-t}\overline{V_r} \gamma g_t  \ne \varnothing$$  for $\gamma \in \Lambda_r$, 
then  {for some ${ h}\in  \overline{V_r}$  one has $g_t{ h}x \in \sigma _{r/2}{O} \cap \overline{V_r} \gamma g_tx$}, 
and, in view of {\equ{bd}} and ${\partial _{r/2}}({\sigma _{r/2}}{O}) \subset O$, we can conclude that 
{$\overline{V_r} \gamma g_t x \subset O$}.} Thus
$${A}^1_x(t,r,{\sigma _{r/2}}{O}) \subset \bigcup_{\substack{\gamma \in {\Lambda_r}\\ \overline{V_r} \gamma g_t x \subset O}}{  {g_{ - t}}\overline{V_r} \gamma {g_t}},$$
and \equ{counting} follows from the definition of $V_r$ being a tessellation domain relative to $\Lambda_r$.
\end{proof}}


\begin{proof}[Proof of Theorem \ref{main cov}]
{Take $a',b, \lambda'$ be as in Proposition \ref{exponential mixing}, and $\lambda_{\min}$ as in \equ{b1}.} Also set
\eq{b'}{b':=\max \left(b,\frac{1}{\lambda'},  {\frac{\log ({16}\sqrt{p})}{\lambda_{\min}} }\right).}
Fix {an open $O \subset X$}, 
and take $r$ as in \equ{r estimate 2}. {Also} take $x \in \partial_r (O^c)$  and 
$t $ as in {\equ{t estimate 2}.}

{First let us show how to derive the desired result for $N=1$ from Proposition \ref{exponential mixing}.} 
Observe that
 $$t \underset{\equ{t estimate 2}}{\ge} a'+b' \log \frac{1}{r} 
 \underset{\equ{r estimate 2}}{\ge} b'    \log \frac{2}{r_*}  \underset{(r_* < \frac14)}> b' \underset{\equ{b'}}{{\ge}}{\frac{\log (8\sqrt{p})}{\lambda_{\min}} }.               $$
{So, by} combining  Lemma \ref{covering} with Lemma \ref{sub count},
 we conclude that 
 %
 ${A}^1_x(t,r,O^c)$ can be covered with at most
{$${
\begin{aligned}
& \# \{ \gamma  \in \Lambda_r :{g_{ - t}}{\overline{V_r}\gamma}{g_t}  \cap \overline{V_r} \ne \varnothing \}  - \# \{ \gamma \in {\Lambda_r} : \overline{V_r} \gamma g_t x \subset O\} \\
& \le  
{e^{\delta t}}  \left(1 + {C_0}e^{-\lambda_{\min} t}\right) - \frac{\nu\left({A}^1_x(t,r,{\sigma _{r/2}}{O})\right)}{\nu({g_{ - t}}{V_r} {g_t})} 
\end{aligned}}$$
Bowen $(t,r)$-boxes in $P$, {where $C_0$ is as in \equ{defd0}}. {Note that {whenever $x \in \partial_r (O^c)$, \equ{rlimit} and  \equ{t estimate} follow from  \equ{r estimate 2}, \equ{t estimate 2} {and \equ{b'}}}. {Moreover, we have 
\eq{lambdat}{\lambda't \underset{\equ{t estimate 2}}{\ge}  \lambda'a'+\lambda'b' \log \frac{1}{r} \underset{\equ{a'},\, \equ{b'}}{\ge} \frac{\log \frac{2}{r_*}}{2} + \log \frac{1}{r} \underset{(r_* < \frac14)}{\ge}    \log \frac{2}{r}.               }
Hence}} one can apply Proposition \ref{exponential mixing} and conclude that ${A}^1_x(t,r,O^c)$ can be covered with at most
$$
\begin{aligned}{e^{\delta t}} \left( {1 +{C_0}e^{-\lambda_{\min} t}-\mu ({\sigma_{e^{-\lambda' t}}\sigma _{r/2}}O) }+ \frac{e^{-\lambda' t}}{\nu (V_r)} \right)
& \underset{\equ{lambdat} }\le  
{e^{\delta t}} \left( {1 +{C_0}e^{-\lambda_{\min} t}-\mu ({\sigma _{r}}O) }+ \frac{e^{-\lambda' t}}{\nu (V_r)} \right) \\
& \le 
{e^{\delta t}}\left( 1-\mu ({\sigma _{r}}O) + \frac{C_2}{r^{p}} e^{-\lambda t} \right)
=:N(r,t) 
\end{aligned}$$
}
Bowen $(t,r)$-boxes in $P,$ where $\lambda:= \min (\lambda_{\min}, \lambda ')$ and \eq{c2}{C_2:= 
{{C_0} + \frac{(4 \sqrt{p})^{p}}{c_1}}.}

Now let $g_{-t}\overline{V_r} \gamma g_t$ be one of the Bowen $(t,r)$-boxes in the above cover which has non-empty intersection with ${A}^1_x(t,r,O^c)$. 
Take any $q={g_{ - t}}{h} \gamma{g_t} \in g_{-t}\overline{V_r} \gamma g_t$; then
${{g_t}qx = {h} \gamma g_t x}$, hence
${\left\{ {{g_t}qx:q \in g_{-t}\overline{V_r} \gamma g_t\,} \right\} = \left\{ {h} \gamma g_t x:{h} \in \overline{V_r} \right\}.}$
Consequently, 
\eq{induction2}{
\{q \in g_{-t}\overline{V_r} \gamma g_t: g_{2t}qx \notin O\}=g_{-t}{{A}^1_x(t,r,O^c)} \gamma g_t.}
Note that since {$\diam(\overline{V_r } \gamma)<r$} and $g_{-t}\overline{V_r} \gamma g_t \cap {{A}^1_x(t,r,O^c)}$ is non-empty, we have $\gamma g_t x \in \partial_{r} (O^c)$. 
Hence, 
by {going through} the same procedure, this time using $\gamma g_t x$ in place of $x$, we can cover the set 
in the left hand side of  \equ{induction2} with at most $N(r,t)$ Bowen $(2t,r)$-boxes in $P$. Therefore, we conclude that the set ${A}^2_x(t,r,O^c)$ can be covered with at most $N(r,t)^2$ Bowen $(2t,r)$-box{es} in $P$. By doing this procedure inductively, we can see that for any $N \in \N$, the set ${A}^N_x(t,r,O^c)$ can be covered with at most
$${N(r,t)^N =  
e^{\delta Nt} \left( {1-\mu ({\sigma _{r}}O) } +\frac{C_2}{r^{p}} e^{-\lambda t} \right)^N}$$ Bowen $(Nt,r)$-box{es} in $P$. {This finishes the proof.}
\end{proof}

{Next} we are going to apply {Theorem \ref{main cov}} to cover $ {{A}^N_x(t,r,O^c)}$ with Bowen $(Nt,\theta)$-boxes, where {$r\le \theta \le \frac {r_*}{2}$}.

\begin{thm}\label{main cor}
Let ${F}$ be a one-parameter {$\Ad$-}diagonalizable
{sub}semigroup of $G,$ and $P$ a subgroup of $G$ that has {property} {\rm (EEP)}. Then, 
with {$a',b'
, C_2, \lambda$} as in   Theorem  \ref{main cov}, 
for any {open $O\subset X$, {any $t$ as in
{\equ{t estimate 2}}, 
any $r$ such that
\eq{r estimate 3}{0<r<\frac{1}{4} \min \Big(r_0\big(\partial_{1} (O^c)\big), r_*\Big),}
any $x\in \partial_{
r}{(O^c)},$ 
any $N \in \N$, and any $\theta\in {\left[{r} ,\frac {r_*}{2}\right]}$}, the set $ {{A}^N_x(t,r,O^c)}$ can be covered with at most}
$${\frac{c_2}{c_1} } {\left(\frac{2r}{\theta} \right)^p} 
{e^{\delta Nt}} \left(1 - \mu \big( {\sigma_{{{4} {\theta}}}} O  \big) +\frac{C_2}{r^{p}} e^{-\lambda t}   \right)^N $$
Bowen $(Nt,\theta)$-boxes in $P$.
\end{thm}

\begin{proof}
{Consider the covering  of $ {{A}^N_x(t,r,O^c)}$  by Bowen boxes  $\left\{g_{-Nt}\overline{V_\theta} \gamma g_{Nt}: \gamma\in\Lambda_\theta\right\}$.
Let $R$ be one of those boxes,  so that  
\eq{nonempty}{R \cap {{A}_x^N\left(t,r,O^c\right)} \neq \varnothing.}
Since $\theta < r_*$, in view of \equ{diam} we have {$\diam (R) \le  \frac{\theta}{2} e^{-\lambda_{\min} Nt }$}; furthermore, 
{\eq{theta-r}{ \theta e^{-\lambda_{\min} t} \underset{{\equ{t estimate 2}}}\le  \theta e^{-\lambda_{\min} b' \log \frac{1}{r}} \le \theta e^{- {\log (8 \sqrt{p})} \log \frac{1}{r}} \underset{(\theta \le 1)}\le  {\frac{r}{8 \sqrt{p}}} .
}}
Since $R\cap V_r \ne \varnothing$, it follows that {$$R \subset \partial_{\frac{\theta}{2} e^{-\lambda_{\min} Nt}}\overline{V_r} \underset{\equ{theta-r}}{\subset} \partial_{{\frac{r}{16 \sqrt{p}}}} \overline{V_r} \subset {V_{2r}},$$}  
where in the last inclusion we again use the $2$-bi-Lipschitz property of $\exp$. 
} 

We now claim that $R$ is  {contained in}
${{A}^N_x\big(t,{2r},{{\partial_{{ {2 \theta}}} (O^c)}}\big)}$. Indeed,  in view of \equ{nonempty} we can find $h_1':= g_{-Nt}h_1 \gamma g_{Nt} \in R$ such that $g_{i t}{h_1'}x \in O^c$ for all $i \in \{1,\dots,N\}$ (here $h_1\in \overline{V_\theta}$).  
Then  take any
 $h_2':=g_{-Nt}h_2 \gamma g_{Nt} \in R$, where again $h_2\in \overline{V_\theta}$,  and for any $i \in \{1,\dots,N\}$ write 
 $${
\begin{aligned}
g_{i t}{h_2'}x
&=(g_{-(N-i)t}{h_2 h_1^{-1}} g_{ (N-i)t})g_{i t}{h_1'}x \\
& \in (g_{-(N-i)t}{h_2 h_1^{-1}} g_{ (N-i)t}) O^c 
 \underset{\equ{diam}}{\subset} \partial_{2 \theta e^{-\lambda_{\min} (N-i)t}}  (O^c)   \subset \partial_{ 2\theta}  (O^c).
\end{aligned} }$$

{Note that since $\theta \le r_*/2  < 1/8$, we have $\partial_{1/2}\big(\partial_{2 \theta} (O^c)\big) \subset \partial_1 O^c$, which implies $r_0\Big(\partial_{1/2}\big(\partial_{2 \theta} (O^c)\big)\Big) \ge r_0(\partial_1 O^c) $. Thus, since \equ{r estimate 3} is satisfied, the following is satisfied as well:
$$0<2r< \frac{1}{2} \min \Big(r_0\big(\partial_{1/2}\big(\partial_{2 \theta} (O^c)\big), r_*\big)\Big) .         $$
Consequently, Theorem \ref{main cov}, applied to $O$ replaced with  {$  \sigma_{{{2}\theta}} O  $ and $r$ replaced with $2r$,} implies that 
\begin{align*}
\nu\Big({{A}^N_x\big(t,{2r},{{\partial_{{ {2 \theta}}} (O^c)}}\big)}\Big)\ &\,\underset{\equ{boundarysigma}}\le \nu\Big({{A}^N_x\big(t,{2r},{{(\sigma_{{ {2 \theta}}} O)^c\big)}}}\Big)\\ & \ \  \le \ \nu(g_{-Nt}{V_{2r}}  g_{Nt}) \cdot {e^{\delta Nt}}\left(1 - \mu \big({\sigma_{2r} (\sigma_{{{{2} \theta}}}} O ) \big) +\frac{C_2}{(2r)^{p}} e^{-\lambda t}  \right)^N 
 \\ & 
 \underset{{(\theta \ge r)} }\le   \nu({V_{2r}} )
 \left(1 - \mu \big( {\sigma_{{{4} {\theta}}}} O  \big) +\frac{C_2}{r^{p}} e^{-\lambda t}  \right)^N
\end{align*}
{for any $x\in \partial_r( O^c) \subset \partial_{2r}\big((\sigma_{{{2}\theta}} O )^c\big)$}.
This forces the number of $\gamma\in\Lambda_\theta$ such that $$g_{-Nt}\overline{V_\theta} \gamma g_{Nt}\cap {{A}_x^N\left(t,r,O^c\right)} \neq \varnothing$$
to be not greater than $\left(1 - \mu \big({ \sigma_{{{4} {\theta}}}} O  \big) +\frac{C_2}{r^{p}} e^{-\lambda t}  \right)^N$ multiplied by
$$
\frac{ \nu({V_{2r}} )}{\nu(g_{-Nt}{V_{\theta}}  g_{Nt})}
 \underset{\equ{lb1}} \le  \frac{ c_2 \big(\frac{2r}{4\sqrt{p}}\big)^p}{e^{-\delta Nt}c_1\big(\frac{\theta}{4\sqrt{p}}\big)^p}
 = {\frac{c_2}{c_1} } {e^{\delta Nt}}{\left(\frac{2r}{\theta} \right)^p}.
$$}
This finishes the proof of the theorem.
\end{proof}

{\ignore{
We will also need to cover Bowen boxes by small balls. 
The next lemma provides a bound for the number of balls of radius $re^{-\lambda' t}$ needed to cover a Bowen $(t,r)$-box.
\begin{lem}
\label{coveringballs} There exists $K_4 > 0$ such that for any $0<r < r_*'$ and any {$t>0 $},  any Bowen $(t,r)$-box in $P$ can be covered with at most  $K_4 \frac{{\nu ({g_{ - t}}V_r {g_{t}})}}{{\nu \left({B^P}(r e^{ -  \lambda' t})\right)}}$
balls in $P$ of radius $re^{- \lambda't}$. 
\end{lem}
\begin{proof} Let $B = g_{-t}V_r \gamma g_t$ be a Bowen $(t,r)$-box. 
In view of the Besicovitch covering property of $P$, any covering of $B$ by balls in $P$ of radius $re^{- \lambda't}$ has a subcovering of  index uniformly bounded from above by a fixed constant   (the Besicovitch constant of $P$). The union of those balls is contained in the $re^{-\lambda' t}$-neighborhood of $B$. But since $B$ is a translate of the exponential image of a box in $\mathfrak p$ whose  smallest sidelength is $re^{-\lambda' t}$, it follows that the measure of  the $re^{-\lambda' t}$-neighborhood of $B$ is bounded by a uniform constant times $\nu(B)$, and the lemma follows.\end{proof}}}

\section{ENDP and iterations of Margulis inequality}\label{iterations}
Suppose $P$ is a subgroup of $G$ which satisfies (ENDP). Then by the definition one can find  $0<c_0<{1}$ and $t_0>0$ such that the following holds: for any $t \ge t_0$ one can find a height function $u_t$ and $d_t>0$ such that {$u_t$ satisfies the 
$(c_0,d_t)$-Margulis inequality with respect to $I_{{B^P(1),t}}u_t$; that is,}
\eq{mi1}{(I_{{B^P(1),t}}u_t)(x) \le c_0 u_t(x)+d_t.}
{{Let} $t_1 >0$ be sufficiently large so that
\eq{conj1}{g_{-{t}} B^P(r)g_{t}  \subset B^P(r/4) \,\,\,  \text{for all} \,\,\, 0<r\le 1,\, t \ge t_1, }  
and set
{\eq{ts}{t_* := \max(t_0, t_1)}}}

In the following proposition, by using inequality \equ{mi1} 
$N$ times for $t$ sufficiently large, we prove that {$u_t$ satisfies the} $\left(c_0^N, 
 {\frac{d_t}{1-c_0}}\right)$-Margulis inequality with respect to $I_{{B^P(1/2),Nt}}$. {The argument is similar to the proof of \cite[Theorem 15]{SS}.}
\begin{prop} \label{ENDP2}
Let $P$ be a subgroup of $G$ that has property {\rm (ENDP)}, and let $\{ u_t\}_{t>0}$ be the family of height functions in the definition of {\rm (ENDP)}. {Also let $t_*$ be as in \equ{ts}.} Then \ignore{there exists $t_1>0$ such that for}{for any $t \ge t_*$} and any $N \in \N$, {$u_t$ satisfies the $\left(c_0^N, 
 {\frac{d_t}{1-c_0}}\right)$ Margulis inequality with respect to $I_{B^P(1/2),\, Nt}$}. In other words, for any {$t \ge t_*$}, { any $N \in \N$ and any $x\in X$ one has}
\eq{infi}{(I_{B^P(1/2),\, Nt}u_t)(x) \le c_0^N u_t(x)+ {\frac{d_t}{1-c_0}}.}
\end{prop}
\smallskip
As a corollary, we get the following crucial statement which will be useful in later sections:
\begin{cor}\label{corh}
Let $P$ be a subgroup of $G$ that has property {\rm (ENDP)}, and {let $t_1$ be as in \equ{conj1}.} Then there exists a height function $u$ and $d>0$ such that for any $0 < c < 1$ one can {find} positive $t_c \ge t_1$ such that for any $ t \in \N t_c$, {$u$ satisfies the $(c,d)$-Margulis inequality with respect to $I_{B^P(1/2),t}$}. In other words, for all $x \in X$ we have
\eq{uin}{(I_{B^P(1/2),\, t}u)(x) \le c u(x)+d.}

\end{cor}
\begin{proof}
 Let $0 < c < 1$, and take $c_0$ as in Proposition \ref{ENDP2}. Choose $N$ sufficiently large so that $c_0^N  \le c$, and set
 $$u:= u_{t_1}, \, t_c:= {Nt_* { \underset{\equ{ts}}{\ge} t_1}}, \, d:=  {\frac{d_{t_1}}{1-c_0}}.$$
 Now let $t=nt_c=nN {t_*}$ be an element in $\N t_c$. 
 Then, by Proposition \ref{ENDP2} applied {with} $N$ replaced {by} $nN$, we have
 $$ (I_{B^P(1/2),\, t}u)(x)=  (I_{B^P(1/2),\, nNt_1}u)(x) \le c_0^{nN} {u(x)}+ d \le c_0^{N} {u(x)}+d \le c u(x)+d .                       $$
 
 This finishes the proof.
\end{proof}
 
\begin{proof}[Proof of Proposition \ref{ENDP2}]
Given $n \in \N$ and $t>0$, define $\eta_{n,t}: B^P(1)^n \rightarrow P$ by 
\eq{eta}{\begin{aligned}
 \eta_{n,t}(h_1,\dots,h_n)&:=g_{-(n-1)t}h_ng_t\cdots h_2g_th_1\\ &\,\, = \tilde{h}_n \cdots \tilde{h}_1, \,\,\, \text{where}\,\, \tilde{h}_i= g_{-(i-1)t}h_i g_{(i-1)t}. \end{aligned}}  
For any $n \in \N$ and $t>0$, let $\tilde{\nu}_{n,t}$ be the 
{pushforward of $\nu|_{B^P(1)}$ via the conjugation by $g_{nt}$, that is,}
defined by
\eq{nun}{ \int_P \phi (h) \,d \tilde{\nu}_{n,t}(h)=\int_{B^P(1)} \phi(g_{-nt}hg_{nt}) \,d \nu(h)    }
for all $\phi \in C_c(P).$ For any positive integer $n$ let $$\nu_{n,t} := \tilde{\nu}_{n-1,t} \ast \cdots \ast \tilde{\nu}_{1,t} \ast \tilde{\nu}_{0,t}$$ be the
measure on $P$ defined by the $n$ convolutions. It is easy to see that  $\nu_{n,t}$ is absolutely continuous
with respect to $\nu$, and $\nu_{n,t}$ is the pushforward of $(\nu|_{B^P
(1)}
)^{\otimes n}$ by the map $\eta_{n,t}$. 
{These measures were considered in \cite{GS}, and the following was shown:}
\begin{lem}\cite [Lemma 5.5]{GS} \label{2lem}
For all $t \ge t_1$ {as in \equ{conj1}}, all $h \in B^P(1/2)$, and for all $n \in \N$ we have $\frac{d \nu_{n,t}}{d \nu}(h) \ge 1$.  
\end{lem}
\smallskip
 {Using Lemma \ref{2lem}, we have for all $N \in \N$ and all $t \ge t_1$:}
\eq{inthalf}
{\begin{aligned}
(I_{B^P(1/2), Nt} u)(x) &= \int_{B^P(1/2)}u(g_{Nt}hx) \,d \nu(h) \le  \int_{B^P(1/2)}u(g_{Nt}hx) \,d {\nu_{N,t}(h)} \\
 & \le { \int_{B^P(1/2)^{N}} 
u({g_{t}h_{{N}} \cdots g_{t} h_{1}x}) \,d \nu^{\otimes n} (h_1, \dots,h_{N}) }\\
& \le  \int_{B^P(1)^{N}} 
u({g_{t}h_{{N}} \cdots g_{t} h_{1}x}) 
\,d \nu^{\otimes n} (h_1, \dots,h_{N})
\end{aligned}}
{Take {$0<c_0<1$} {and $t_0>0$} as in the definition of (ENDP), and let $t \ge t_*= \max(t_0,t_1)$. Recall that $\nu\left(B^P(1)\right)=1$. Since $t \ge t_0$, we can apply \equ{mi1} and  for any 
{$i=2,\dots$ conclude that}
\eq{second-in}{\begin{aligned}
&
 \int_{B^P(1)^{i}} 
u({g_{t}h_{{i}} \cdots g_{t} h_{1}x}) 
\,d \nu^{\otimes i} (h_1, \dots,h_{i}) \\
& \le   \int_{B^P(1)^{i-1}} 
\left(c_0 \cdot u({g_{t}h_{{i-1}} \cdots g_{t} h_{1}x}) + d_t \right) \,d \nu^{\otimes {i-1}} (h_1, \dots,h_{i-1})\\
& = c_0  \int_{B^P(1)^{i-1}} u({g_{t}h_{{i-1}} \cdots g_{t} h_{1}x})\,d \nu^{\otimes {i-1}} (h_1, \dots,h_{i-1}) + d_t \cdot \nu\left(B^P(1)^{i-1}\right) \\ 
&  =   c_0  \int_{B^P(1)^{i-1}} u({g_{t}h_{{i-1}} \cdots g_{t} h_{1}x})\,d \nu^{\otimes {i-1}} (h_1, \dots,h_{i-1}) + d_t.   
\end{aligned}}
Let $N \in \N$. If $N=1$, then \equ{infi} follows immediately from the combination of \equ{mi1} and \equ{inthalf}. If $N \ge 2$, then by using \equ{second-in} repeatedly and combining with \equ{inthalf} we obtain
\eq{linf}{
\begin{aligned}
 (I_{B^P(1/2), Nt} &u)(x)          
 \le \int_{B^P(1)^{N}} 
u({g_{t}h_{{N}} \cdots g_{t} h_{1}x})\,d \nu^{\otimes n} (h_1, \dots,h_{N}) \\
& \le c _0^{N-1} \int_{B^P(1)} 
u(g_{t} h_{1}x)\,d \nu(h_1) + c_0^{N-2}d_t+ \cdots + c_0 d_t +d_t \\
& \underset{\equ{mi1}}\le  c _0^N u(x) + c_0^{N-1}d_t+ c_0^{N-2}d_t+ \cdots + c_0 d_t +d_t      \\
& < c_0^N u(x)+d_t   (1+c_0+ c_0^2 + \cdots) 
 = c_0^N u(x) + \frac{d_t}{1-c_0} .
\end{aligned}
}
}
This finishes the proof.
\end{proof}

\section{(ENDP) and escape of mass}\label{escape}
Let
$P$ be a subgroup of $G$ that has property (ENDP). Take a height function $u$, and  for $M>0$ define the following sets:
$$X_{> M}:= \{x\in X :u(x)> M\},  \,\,\,\,\, X_{ \le M} :=  \{x\in X :u(x) \le  M\}.$$
Since $u$ is proper, the sets $X_{\le M}$ are compact.  

 Since $u$ is regular, by definition there exists $C \ge 1$ such that 
 {\eq{reg1}{C^{-1}u(x) \le u(gx) \le C u(x) \text{ for all $g \in B(2)$ and }x \in X.}}
Moreover,  it is easy to see {from  \equ {reg1}} that there exists $\alpha>0$ such that for any $t>0$ we have
 \eq{expand}{e^{- \alpha t} u(x) \le u(g_tx) \le e^{\alpha t} u(x)}
Now 
let $0 < c < 1$, take $d$ and $t_c \ge t_1$ as in Corollary \ref{corh}, and let $t \in \N t_c$. Note that \equ{uin} immediately implies that if {$u(x) \ge \frac{d}{c},$} then 
\eq{2ine}{(I_{B^P(1/2),t}u)(x) \le 2c \cdot u(x)  }
Now define
\eq{lt}{\ell_{c,t}:= \max \left(\frac{d}{c}, e^{\alpha t}\right)}
In the following key proposition, we obtain {an} upper bound for the measure of the sets of type $A_x^N\left(kt,\theta,X_{>   C^2 \ell_{c,t}^2}\right)$, where $2 \le k \in \N$, $\theta \in (0,{r_*]}$, and $C$ is as in \equ{reg1}. We will use this measure estimate to derive a covering result for the sets of type $A_x^N\left(kt,\theta,X_{>   C^3 \ell_{c,t}^2}\right)$ in Corollary \ref{fin cor}.
 
\begin{prop}\label{mest}
For any $2 \le k \in \N$, any $\theta \in (0, {r_*]}$, any $N \in \N$, and for any $x \in X$ we have
\eq{main-measure}{\nu \left(A_x^N\big(kt,\theta,X_{>   C^2 \ell_{c,t}^2}\big) \right) \le \left({\frac{4c}{1-c}} \right)^{N}   \frac{\max\big(u(x),d \big)}{\ell_{c,t}^2} }
\end{prop}
\begin{proof}[Proof of Proposition \ref{mest}]
Let ${2 \le k \in \N}$, $N \in \N$ and 
$x \in X$. Define
$$Z_{x}(k,N):=   \left\{(h_1, \dots, h_{Nk}) \in B^P(1/2)^{Nk}: {{u}} (g_th_{nk} \cdots g_t h_1x) > C \ell_{c,t}^2 \,\,\, \forall n \in \{1, \dots,N\}\right\}.    $$
We need the following lemma:
\begin{lem}\label{1lem}
For all $\theta \in (0,{r_*]}$ and for all $h \in A_x^N\big(kt,\theta,X_{>   C^2 \ell_{c,t}^2}\big)$ one has \linebreak $\eta^{-1}_{Nk,t}(h) \subset Z_{x}(k,N)$, where $\eta_{Nk,t}$ is defined as in \equ{eta}.
\end{lem}

{\begin{proof}
Let $\theta \in (0,{r_*]}$ and let $h \in A_x^N\big(kt,\theta,X_{>   C^2 \ell_{c,t}^2}\big)$. Suppose
that $$\eta_{Nk,t}(h_1,\dots, h_{Nk})=h.$$
Then for any $1 \le i \le N$ we have 
{
$$
\begin{aligned}
 \dist(g_{ikt}h, g_t h_{ik} \cdots g_t h_1) &\underset{\equ{eta}}{=}   \dist ({g_{ikt} \tilde{h}_{Nk} \cdots \tilde{h}_{1} } ,  g_{ikt}  \tilde{h}_{ik} \cdots \tilde{h}_{1})
 \\
 &\ \  {=} \ \, \begin{cases}\dist ({g_{ikt} \tilde{h}_{Nk} \cdots \tilde{h}_{ik+1} g_{-ikt}} , e)&\text{ if } i < N,\\ 0&\text{ if } i = N.\end{cases}
 \end{aligned}
 $$
 Moreover, if $i < N$ one has $$
\begin{aligned}
\dist ({g_{ikt} \tilde{h}_{Nk} \cdots \tilde{h}_{ik+1} g_{-ikt}} , e)& \ \  \le  \ \, \dist(g_{ikt}\tilde{h}_{ik+1} g_{-ikt},e) + \cdots + \dist(g_{ikt}\tilde{h}_{Nk} g_{-ikt}, e) \\
& \underset{\equ{eta}}{=}\dist(h_{ik+1},e)+ \dist (g_{-t}h_{ik+2}g_t, e)+ \dist(g_{-2t}h_{ik+3}g_{2t}, e)\\
& \ \  {+} \ \,  \cdots +\,    \dist (g_{-((N-i)k-1)t}h_{Nk}g_{((N-i)k-1)t}, e) \\
& \underset{\equ{conj1}}\le  1+ \frac{1}{4}+ \frac{1}{4^2}+\cdots \frac{1}{4^{(N-i)k-1}} < 2.
\end{aligned}$$}
Hence, {in view of  \equ{reg1}}, for any $1 \le i \le N$, $g_{ikt}h{x} \in X_{>   C^2 \ell_{c,t}^2}$ implies {that} $ g_t h_{ik} \cdots g_t h_1{x}  \in X_{>   C \ell_{c,t}^2}$. This finishes the proof.
\end{proof}}
 
Now let $\theta \in (0,{r_*]}$. {Note that in view of \equ{Bowen inc} we have $\overline{V_\theta} \subset {B^P(r_*/2)} \subset B^P(1/2)$; moreover, $kt \ge  k t_c \ge t_1$. Thus,} by Lemma \ref{1lem} and Lemma \ref{2lem} we have:
\eq{measure-in}{\begin{aligned}\nu \left(A_x^N\big(kt,\theta,X_{>    C^2 \ell_{c,t}^2}\big) \right) &\le \nu_{Nk, \, t} \left(A_x^N\big(kt,\theta,X_{>   C^2 \ell_{c,t}^2}\big) \right)\\ & \le \nu^{\otimes Nk}\big(Z_{x}(k,N)\big),\end{aligned}}
where $\nu_{Nk,\, t}$ is defined as in \equ{nun}. So it suffices to estimate $\nu^{\otimes Nk}\big(Z_{x}(k,N)\big)$.
{Define}
$$s(k,N,x):=\int_{Z_{x}(k,N)}  u(g_th_{Nk} \cdots g_th_1x) \,d \nu^{\otimes Nk} (h_1, \dots,h_{Nk}).      $$ 
Since $\ell_{c,t} \ge e^{\alpha t}$,
 in view of \equ{reg1} and \equ{expand} {we have   $u(g_th_{k-1} \cdots g_th_{{1}}x) > \ell_{c,t}$ whenever $(h_1, \dots, h_{k}) \in Z_x(k,1)$}. Hence,
\eq{first-in}{\begin{aligned}
 s(k,1,x)
&\le \int_{B^P(1/2)^{k-1}} 1_{
{X_{> \ell_{c,t}}}
}(g_{t}h_{{k-1}} \cdots g_t h_{1}x)  u(g_th_{k} \cdots g_th_{{1}}x)  \,d \nu^{\otimes k-1} (h_1, \dots,h_{k-1})  \\
& \le 2 c \int_{B^P(1/2)^{k-1}} 
u({g_{t}h_{{k-1}} \cdots g_t h_{1}x})   \,d \nu^{\otimes k-1} (h_1, \dots,h_{k-1}),
\end{aligned}}
{where the second inequality  follows from \equ{2ine} applied with $x$ replaced {by} $g_th_{k-1} \cdots g_th_1x$, and {from} the fact that $\ell_{c,t} \ge {\frac d c}$.}

{Again recall that  $\nu\left(B^P(1)\right)=1$. By applying \equ{uin} we get:
$${\begin{aligned}
&
 \int_{B^P(1/2)^{k-1}} 
u({g_{t}h_{{k-1}} \cdots g_t h_{1}x}) 
\,d \nu^{\otimes k-1} (h_1, \cdots,h_{k-1}) \\
&  \le c  \int_{B^P(1/2)^{k-2}} 
u({g_{t}h_{{k-2}} \cdots g_t h_{1}x}) 
\,d \nu^{\otimes k-2} (h_1, \cdots,h_{k-2})   + d \cdot \nu \left( {B^P(1/2)}^{k-2} \right)                             \\
& \le c  \int_{B^P(1/2)^{k-2}} 
u({g_{t}h_{{k-2}} \cdots g_t h_{1}x}) 
\,d \nu^{\otimes k-2} (h_1, \cdots,h_{k-2})   + d. 
\end{aligned}
}$$
Therefore, if we apply \equ{uin} repeatedly, similarly to \equ{linf} we get
\eq{second-in1}{\begin{aligned}
&  \int_{B^P(1/2)^{k-1}} 
u({g_{t}h_{{k-1}} \cdots g_t h_{1}x}) 
\,d \nu^{\otimes k-1} (h_1, \cdots,h_{k-1}) \\
& \le c^{k-1} u(x)+\frac{d}{1-c}   \le \frac{2}{1-c} \cdot \max\big(u(x),d\big).
\end{aligned}}
}
So by combining \equ{first-in} and \equ{second-in1} we have:
\eq{mainind}{s(k,1,x) \le   {\frac{4c}{1-c}} \cdot \max\big(u(x),d \big) \,\,\, \text{for all } x \in X.}
Note that since $\ell_{c,t} \ge e^{\alpha t},$ in view of {\equ{reg1}, \equ{expand}} and \equ{lt} 
\eq{implic}{(h_1, \dots,h_{ik}) \in Z_x(k,i) \Rightarrow u(g_th_{(i-1)k} \cdots g_th_{{1}}x) \ge \ell_{c,t} \ge \frac{d}{c}  \ge d.}
Now for any $2 \le i \in \N$ we can write
\begin{equation*}
{\begin{aligned}
& s(k,i,x)= \int_{Z_x(k,i)}u(g_th_{ik} \cdots g_th_{{1}}x) \,d \nu^{\otimes ik} (h_1, \dots,h_{ik}) \\
& = \int_{Z_x(k,i-1)} \int_{Z_{g_th_{(i-1)k} \cdots g_th_{{1}}x}(k,1)}u(g_th_{ik} \cdots g_th_{{1}}x)\,d \nu^{\otimes k} (h_{(i-1)k+1}, \dots,h_{ik})\,d \nu^{\otimes (i-1)k} (h_{1}, \dots,h_{(i-1)k}) \\
& = \int_{Z_x(k,i-1)} s(k,1,g_th_{(i-1)k} \cdots g_th_{{1}}x) \,d \nu^{\otimes (i-1)k} (h_{1}, \dots,h_{(i-1)k}) \\
& \underset{\equ{mainind}, \, \equ{implic}}\le  \int_{Z_x(k,i-1)} {\frac{4c}{1-c}} \cdot u(g_th_{(i-1)k} \cdots g_th_{{1}}x)\,d \nu^{\otimes (i-1)k} (h_{1}, \dots,h_{(i-1)k}) \\
& ={\frac{4c}{1-c}}  \cdot s(k,i-1,x)
\end{aligned}}
\end{equation*}
 Thus, {by repeatedly using  
 the above computation,  for any $N \in \N$ we conclude that}
 $$
  s(k,N,x) \le \left({\frac{4c}{1-c}} \right)^{N-1} s(k,1,x)  \underset{\equ{mainind}}\le  \left({\frac{4c}{1-c}} \right)^{N} \max\big(u(x), d\big).
 $$
 Note that $s(k,N,x) \ge \ell_{c,t}^2 \cdot \nu^{\otimes Nk}\big(Z_x(k,N)\big)$. Hence \equ{main-measure} follows from the above inequality and \equ{measure-in}.
\end{proof}
As a corollary, we get the following crucial covering result:
\begin{cor}\label{fin cor}
Let $P$ be a subgroup of $G$ {with property} {\rm (ENDP)}. 
Then  for any $0 < c < 1$ there exists $t_c>0$ such that {{for} all $t \in \N t_c$ and $2 \le k \in \N$ satisfying {$kt \ge \frac{\log (8 \sqrt{p})}{\lambda_{\min}}$},{ all $\theta \in {(0,r_*/2]}$,}   all $N \in \N$,  and for all  $x \in X$,} the set 
$${A_x^N\big(kt,\theta,X_{>   C^3  \ell_{c,t}^2}\big) = \left\{  h \in \overline{V_\theta}:u(g_{ikt} hx) > C^3  \ell_{c,t}^2 \,\,\, \forall\, i \in \{1,\dots, N  \} \right\}}$$ 
can be covered with at most 
$$\frac{{e^{\delta Nkt} }}{{\nu({V_\theta}) }}  {\left({\frac{4c}{1-c}}\right)^{N}}  \frac{\max\big(u(x),d \big)}{\ell_{c,t}^2}$$
Bowen $(Nkt, \theta)$-balls in $P$.  
\end{cor}

\begin{proof}
{Let $0 < c < 1$, take $t_c$ as in Corollary \ref{corh}, {and let $t \in \N t_c$ and $2 \le k \in \N$ be such that {$kt \ge \frac{\log(8 \sqrt{p})}{\lambda_{\min}}$.} Also let {$\theta \in (0,{r_*/2]}$,} $N \in \N$,  and   $x \in X$.}}
Take a covering of 
{$ \overline{V_\theta}$} with Bowen  $(Nkt, \theta)$-boxes in 
{$ P$}. Now let {$R$} be one of the Bowen boxes in this cover which has non-empty intersection with $
{A_x^N\big(kt,\theta,X_{>   C^3  \ell_{c,t}^2}\big)} $. {Note that in view of \equ{diam}, we have 
$$\diam(R) \le \frac{\theta}{2} e^{-\lambda_{\min} Nkt} \le \frac{\theta}{2} e^{-\lambda_{\min} kt} \le \frac{\theta}{16 \sqrt{p}}.$$
So, since $R \cap \overline{V_\theta} \neq \varnothing$, we must have
\eq{R-inc}{R \subset \partial_{\frac{\theta}{16 \sqrt{p}}} \overline{V_\theta} \subset {V_{2 \theta}},}
where in the last inclusion we use the $2$-bi-Lipschitz property of $\exp$.}

{Now} let ${h} \in {{R} \cap 
A_x^N\big(kt,\theta,X_{>   C^3  \ell_{c,t}^2}\big)} $. Then {$$u(g_{ikt}{h}x)> C^3  \ell_{c,t}^2 \text{   for all }1 \le i \le N.$$} {On the other hand},  {if we denote the center of $R$ by $h_0$, then for all $1 \le i \le N$ we have for all $h' \in R$:
$${
\begin{aligned}
g_{ikt}h' x 
&  = \left(g_{ikt}h'h_0^{-1}g_{-ikt}\right)g_{ikt}h_0x \\
& \in \left(g_{-(N-i)kt}\overline{V_\theta} g_{(N-i)kt}\right) g_{ikt}h_0 {x}\\
& \underset{\equ{diam}}{\in} B \left( {\frac {\theta}{2} \cdot e^{-\lambda_{\min} (N-i)kt}} \right)g_{ikt}h_0x \\
& \subset B ( {\theta}/{2})g_{ikt}h_0x \subset B ( 1/2)g_{ikt}h_0x
\end{aligned}}$$}
{This implies that
\eq{h0eq}{g_{ikt}Rx \subset B(1) g_{ikt}hx  \text{   for all }1 \le i \le N               }} 
Now in view of \equ{R-inc} and {\equ{h0eq}} we can conclude that
\eq{cube inc}{{R} \subset  
{A_x^N\big(kt,{2\theta},X_{>   C^2  \ell_{c,t}^2}\big).}}
{Therefore, by \equ{main-measure} and \equ{cube inc} applied with $\theta$ replaced with $2 \theta$}, the set $
{A_x^N\big(kt,{ \theta},X_{>   C^3  \ell_{c,t}^2}\big)}$ can be covered with at most  
$$ \begin{aligned}
\frac{{\nu}{\left(
A_x^N\big(kt,{2\theta},X_{>   {C^2}  \ell_{c,t}^2}\big)\right)}}{\nu(g_{-Nkt} {V_{{ \theta}}} g_{Nkt})} 
& \le \frac{({\frac{4c}{1-c}})^{N} \max\big(u(x),d\big)}{\nu(g_{-Nkt}{V_\theta} g_{Nkt}) \cdot \ell_{c,t}^2} \\
& = {\frac{{e^{\delta Nkt} }}{{\nu({V_\theta}) }}   {\left({\frac{4c}{1-c}}\right)^{N}}   \frac{\max\big(u(x),d \big)}{\ell_{c,t}^2}} 
\end{aligned}$$
Bowen  $(Nkt, \theta)$-boxes in 
{$ P$}. This finishes the proof.
\end{proof}

\section{
Combining the estimates of \S\ref{boxes} and \S\ref{escape} }\label{abstractlemma}
The goal of this section is to describe a method making it possible to put together properties (EEP) and (ENDP).  In the next proposition  nether (EEP) nor (ENDP) are assumed to hold. Instead we will assume certain covering estimates (similar to those we derived from (EEP) and (ENDP) in the previous sections) and then combine them to derive an estimate on which our dimension bound is based. This  formalizes an argument which first appeared in \cite{KKLM} and then was used in \cite{KMi2} to solve DDC in the case \equ{gt}.  

\begin{prop}\label{abstract} 
Let 
$P$ be a {connected} subgroup of $H$ normalized by $F$. Let $S,Q\subset X$, {$t$ satisfying \equ{bigt}},
$r > 0$, ${\theta \in[ r, r_*/2]}$, 
and let $k_1,k_2,a_1,a_2 \ge 1$ be given. Suppose that for any $N\in \N$ the following two conditions hold:
\begin{itemize}
\item[{\rm (a)}] For all $x\in \partial_r(S\cap Q)$ the set $A_x^N(t,r,S\cap Q)$ can be covered with at most 
$k_1 {e^{\delta Nt}}a_1^N    $
Bowen $(Nt,\theta) $-boxes in $P$.
\item[{\rm (b)}] 
For all $x\in \partial_\theta (S \cap Q)$ the set $A_x^N(t,\theta,Q^c)$ can be covered with at most 
$k_2 {e^{\delta Nt}}a_2^N    $
Bowen $(Nt,\theta) $-boxes in $P$.
\end{itemize}
Then for all $x\in \partial_r(S\cap Q)$ the set ${A}^N_x(t,r,S)$ can be covered with at most  \linebreak
$k_3 {e^{\delta Nt}}a_3^N    $
Bowen $(Nt,\theta) $-boxes in $P$, where
\eq{a3}{k_3 ={(1+C_0) {\frac{c_2}{c_1} \left( 
{\frac \theta r + 8\sqrt p }   \right)^p} k_1} k_2^2 ,\quad a_3 = a_1+a_2 +\sqrt{k_3a_2}.
}
\end{prop}

\begin{proof}
 For any $h \in {A_x^N \left(t,r,S\right)}$, let us define:
$${J_h}:=\big\{j \in \{1,\dots,N\}:g_{jt}hx \in Q^c\big\},$$ and for any $J \subset \{ 1,\dots,N  \}$, set:
$$Z(J):=\left\{ h \in {A_x^N \left(t,r,S\right)}: {J_h}=J \right\}.$$
Note that
\eq{union1}{          {A_x^N \left(t,r,S\right)} =\bigcup_{ J \subset \{   1,\dots,N\} } Z(J)}
 Let $J$ {be a subset of} $ \{1,\dots, N   \}$. We can decompose  $J$ and $I:=\{1,\dots,N\}  \ssm J$  into sub-intervals of maximal size $J_{1}, \dots,J_{q}$ and $I_{1}, \dots,I_{q'}$ so that
$$J=\bigsqcup_{j=1}^{q} J_{j} \text{ and }I=\bigsqcup_{i=1}^{q'} I_{i}.$$ Hence, we get a partition of the set $\{1,\dots,N\}$ as follows:
$$\{1,\dots,N \}=  \bigsqcup_{j=1}^{q} J_{j} \sqcup   \bigsqcup_{i=1}^{q'} I_{i}                  .$$ 
Now we inductively prove 
{the following} 

\begin{claim}\label{claim1}
For any integer $L \le N$, if
\eq{induc eq1}{\{1,\dots,L\}= \bigsqcup_{j=1}^{\ell} J_{j}
\sqcup \bigsqcup_{i=1}^{\ell '} I_{i},}
then the set $Z(J)$ can be covered with at most
\eq{L case11}{
k_2^{d'_{J,L}+1}\left({(1+C_0) {\frac{c_2}{c_1}\left({\frac \theta r + 8\sqrt p }    \right)^p} k_1} \right) ^{d_{J,L}+1}{e^{\delta Lt}}
  a_1^{\sum_{i=1}^{\ell'} |I_i|-d_{J,L}} 
   a_2^ {\sum_{j=1}^{\ell } |J_j|}
} 
Bowen $(Lt,\theta)$-boxes in $P,$ where $d_{J,L}$, $d'_{J,L}$ are defined as follows:
$$d_{J,L}:=\# \{i \in \{1,\dots, L\}: \ i<L, \,i \in J  \  and\  i+1 \in I\},$$ 
$$d'_{J,L}:=\# \{i \in \{1,\dots, L\}:\ i<L,\,i \in I \  and\  i+1 \in J \}.$$ 
\end{claim}

\begin{proof}[Proof of Claim \ref{claim1}] We argue by induction on $\ell + \ell'$. 
When $\ell + \ell' = 1$, we have   $d_{J,L}=d'_{J,L}=0$, and
there are two cases: either $\ell = 1$ and $\{1,\dots,L\}=J_{1}$, or  $\ell' = 1$ and $\{1,\dots,L\}=I_{1}$.
In the first case
$$
\begin{aligned} & Z(J)   \subset  A_x^L \left(t,{r},Q^c\right) \subset  A_x^L \left(t,\theta,Q^c\right),\end{aligned}$$
Therefore, condition (b) applied with $N = L$
implies that this set  can be covered with at most
$$\begin{aligned}
& k_2 e^{\delta L t} a_2^L 
 < {k_1} k_2 {(1+C_0) { \frac{c_2}{c_1}  \left( {\frac \theta r + 8\sqrt p }    \right)^p}}  e^{\delta L t} a_2^L
 \end{aligned}$$
Bowen $(Lt,\theta)$-boxes in $P$. This finishes the proof of the first case.

In the second case, note that
$$
Z(J) \subset  A_x^L \left(t,{r},S \cap Q\right).$$
Moreover, by condition (a) applied with $N=L$,  $A_x^N  \left(t,{r},S \cap Q \right)$ can be covered by at most
$$
k_1 e^{\delta Lt} a_1^L < {k_1} k_2 {(1+C_0) { \frac{c_2}{c_1}  \left( {\frac \theta r + 8\sqrt p }    \right)^p}}  e^{\delta L t}  a_1^L$$
Bowen $(Lt,\theta)$-balls in $P$. This ends the proof of the base of the induction.

\smallskip
\ignore{In the first step of the induction, if $\{1,\dots,L\}=I_{1}$, we have $\sum_{j=1}^{\ell '} |J_j|=0$, $d_{J,L}=0$, and $d'_{J,L}=0$. 
Furthermore, by definition
$$ 
\begin{aligned}
&\left\{ s \in M_{m,n}^{{r}/{32 \sqrt{L} }}: g_{\ell kt}h_sx \in  O^c \cap {Q_{c,t}} \,\,\, \forall\,  \ell \in I_1 \right\} \\
& = {{A}_x \left(kt,{r},{L}, {O^c} \cap Q_{c,t}\right). }              \end{aligned} $$
Thus, the claim in this case follows from Theorem \ref{main cor} applied with {$S$ replaced with $O^c \cap Q_{c,t}$}, $N$ replaced by $L$, and $t$ replaced with $kt$.\\ Also if $\{1,\dots,L\}=J_{1}$, we have $  \sum_{i=1}^\ell |I_i|=0, d_{J,L}=0$, and $d'_{J,L}=0$. Moreover, since ${r \underset{\equ{ineq beta3}}\le  1}$, we get
$$
\begin{aligned}
\left\{ s \in M_{m,n}^{{r}/{32 \sqrt{L} }}: g_{\ell kt}h_sx \in  {Q_{c,t}}^c \,\,\, \forall\,  \ell \in J_1 \right\}  
&  \subset A_x\left(kt,1,L,Q_{c,t}^c\right)\\
& \underset{\equ{qt}}{=}  A_x\left(kt,1,L,X_{>  C_{\alpha}^3 e^{mnt}}\right)
.  
\end{aligned}
$$
Hence, the claim in this case follows from Corollary \ref{fin cor} applied with {$M=C_{\alpha}^3 e^{mnt}$} and $N$ replaced with $L$.\bigskip \\}
In the inductive step, let $L'>L$ be the next integer for which an equation similar to \equ{induc eq1} is satisfied. We have two cases. Either
\eq{case11}{\{1,\dots, L'\}=\{1,\dots,L\} \sqcup I_{\ell'+1}}
or
\eq{case21}{\{1,\dots, L'\}=\{1,\dots,L\} \sqcup J_{\ell+1}.}
We start with the case \equ{case11}. Note that in this case we have
\eq{d11}{d_{J,L'}=d_{J,L}+1 \text{ and } d'_{J,L'}=d'_{J,L}.}
{
{By the induction hypothesis, an upper bound for the number of Bowen $(Lt,\theta)$-boxes needed to cover $Z(J)$ is given by \equ{L case11}. Then observe that:
\begin{itemize}
\item In view of \equ{bigt} and Lemma \ref{covering},
\eq{bound1}{e^{\delta t}(1+C_0e^{-\lambda_{\min} kt}) \le e^{\delta t} (1+C_0)}
is an upper bound for the number of Bowen $\big((L+1)t,\theta\big)$-boxes needed to cover an arbitrary Bowen $(Lt,\theta)$-box;
\item In view of Lemma \ref{coveringtheta},\
\eq{bound2} {\frac{c_2}{c_1}\left( {\frac \theta r + 8\sqrt p }    \right)^p}
is an upper bound for the number of Bowen $\big((L+1)t,r\big)$-boxes needed to cover an arbitrary Bowen $(Lt,{\theta})$-box.
\end{itemize}}
Now let $B_r$ be a Bowen $\big((L+1)t,r\big)$-box that has non-empty intersection with $Z(J)$, and let $h \in B_r \cap Z(J)$. Since $h \in Z(J)$, it follows that $g_{(L+1)t}hx \in  S \cap Q$. Therefore, if we denote the center of $B_r$ by $h_0$, we have
\eq{h01}{{g_{(L+1)t}h_{0}x \in \overline{V_r}(S \cap Q) \subset \partial_{r} (S \cap Q).}} Moreover, for any $h \in B_r$ and any positive integer $1 \le i \le L'-(L+1)$ we have:\
$$\begin{aligned}
g_{(L+1+i)t}hx
&=g_{it}(g_{(L+1)t}hh_0^{-1}g_{-(L+1)t})(g_{(L+1)t}h_{0}x).
\end{aligned}$$
{Since} the map $h \rightarrow g_{(L+1)t}hh_0^{-1} g_{-(L+1)t}$ {sends} $B_r$ into $\overline{V_r}$, 
{the preceding equality implies that}
$$
\begin{aligned}
& \left \{h' \in B_r: g_{(L+1+i)t}h'x \in S \cap Q \,\,\, \forall \, i \in \{1,\dots, L'-(L+1) \} \right \}  \\
& \subset  g_{-(L+1)t} A_{g_{(L+1)t}h_{0}x}^{L'-(L+1)}\left(t,r,S \cap Q \right)g_{(L+1)t}h_0. 
\end{aligned}
$$ \smallskip
So, in view of the above inclusion and \equ{h01}, we can go through the same procedure and apply condition (a) with $N$ replaced with $|I_{\ell'+1}|-1=L'-(L+1)$ and $x$ replaced with $g_{(L+1)t}h_{0}x$, and conclude that $B_r \cap Z(J)$ can be covered with at most \eq{bound3}{k_1 e^{\delta (|I_{\ell'+1}|-1)t}a_1^{|I_{\ell'+1}|-1}} Bowen $(L't,\theta)$-{boxes} in $P$.
Multiplying the bounds  \equ{L case11}, \equ{bound1}, \equ{bound2} and  \equ{bound3}, we conclude
that $Z(J)$ can be covered with at most  
\begin{align*}
&\frac{c_2}{c_1}   \left({\frac \theta r + 8\sqrt p } \right)^{p} {e^{\delta (|I_{{\ell'} +1}|-1)t}}a_1^{|I_{{\ell'}+1}|-1} {(1+C_0)} \\
&  \cdot  k_2^{d'_{J,L}+1}\left({(1+C_0) {\frac{c_2}{c_1} \left( \frac{2 \theta}{r}    \right)^p} k_1} \right) ^{d_{J,L}+1}{e^{\delta (L+1)t}}
  a_1^{\sum_{i=1}^{\ell'} |I_i|-d_{J,L}} 
   a_2^ {\sum_{j=1}^{\ell } |J_j|} \\
& \underset{\equ{d11}}{=} k_2^{d'_{J,L'}+1}\left({(1+C_0) {\frac{c_2}{c_1}\left({\frac \theta r + 8\sqrt p }    \right)^p} k_1} \right) ^{d_{J,L'}+1}{e^{\delta L't}}
a_1^{\sum_{i=1}^{\ell'} |I_i|-d_{J,L'}}
a_2^ {\sum_{j=1}^{\ell } |J_j|} 
\end{align*}
Bowen $(L't,\theta)$-{boxes} in $P$.
This ends the proof of the claim in this case. 
\smallskip

Next assume  
\equ{case21}.
Note that in this case
\eq{d21}{d_{J,L'}=d_{J,L}\text{ and } d'_{J,L'}=d'_{J,L}+1.}
Take a covering of $Z(J)$ with Bowen $(Lt,\theta)$-boxes in $P$. Suppose $B'$ is one of the Bowen $(Lt,\theta)$-boxes in the cover such that $B' \cap Z(J) \neq \varnothing$, and let $h_1$ be the center of $B'$. It is easy to see that $B' \cap Z(J) \neq \varnothing$ implies:
\eq{h11}{g_{Lt}h_{1}x \in \overline{V_\theta}(S \cap Q) \subset \partial_{\theta} (S \cap Q).}
On the other hand, for any $s \in B'$ and any positive integer $1 \le i \le L'-L$ we have:
$$\begin{aligned}
g_{(L+i)t}h_1x 
&=g_{it}(g_{Lt}hh_1^{-1} g_{-Lt})(g_{Lt}h_{1}x) .
\end{aligned}
$$
Hence, since the map $h \rightarrow g_{Lt}{hh_1^{-1}}g_{-Lt}$ maps $B'$ into $\overline{V_{\theta}}$, 
the above equality implies
$$
\begin{aligned}
 \left \{h \in B': g_{(L+i)t}hx \in Q^c \,\,\, \forall \, i \in \{1,\cdots, L'-L \} \right \}  
 \subset  g_{-Lt} A_{g_{Lt}h_{1}x}^{L'-L}\left(t,\theta,Q^c \right)g_{Lt} h_1
\end{aligned}
$$
So in view of the above inclusion and \equ{h11}, we can apply condition (b) with  $g_{Lt}h_{1}x$ in place of $x$, and $|J_{\ell+1}|=L'-L$ in place of $N$. This way, we get that the set $B' \cap Z(J) $ can be covered with 
at most
$ k_2 a_2^{|J_{\ell+1}|}  {e^{\delta |J_{\ell+1}|t}}
$ 
Bowen $(L't,\theta)$-boxes in $P$.
From this, combined with the induction hypothesis, we conclude that $Z(J)$ can be covered with at most 
$$
\begin{aligned}
& k_2 a_2^{|J_{\ell+1}|}  {e^{\delta |J_{\ell+1}|t}}   \cdot  k_2^{d'_{J,L}+1}\left({(1+C_0) {\frac{c_2}{c_1} \left( {\frac \theta r + 8\sqrt p }    \right)^p} k_1} \right) ^{d_{J,L}+1}{e^{\delta Lt}}
  a_1^{\sum_{i=1}^{\ell'} |I_i|-d_{J,L}} 
   a_2^ {\sum_{j=1}^{\ell } |J_j|} \\
& \underset{\equ{d21}}{=}k_2^{d'_{J,L'}+1}\left({(1+C_0) {\frac{c_2}{c_1} \left( {\frac \theta r + 8\sqrt p }   \right)^p} k_1} \right) ^{d_{J,L'}+1} {e^{\delta L't}}
a_1^{\sum_{i=1}^{\ell'} |I_i|-d_{J,L}} 
a_2^ {\sum_{j=1}^{\ell +1} |J_j|}
\end{aligned}
$$
Bowen $(L't,\theta)$-boxes in $P$,}
finishing the proof of the claim. \end{proof}

Now by letting $L=N$, we conclude that $Z(J)$ can be covered with at most \eq{last step1}{
 k_2^{d'_{J,N}+1}\left({(1+C_0) {\frac{c_2}{c_1}\left( {\frac \theta r + 8\sqrt p }   \right)^p} k_1} \right) ^{d_{J,N}+1}{e^{\delta Nt}}
  a_1^{|I|-d_{J,N}} 
   a_2^ {|J|} }
Bowen $(Nt,\theta)$-{boxes} in $P$. 

Clearly \eq{lastre1}{d'_{J,N} \le d_{J,N}+1.} Also, note that since $d_{J,N} \le \max (|I|,|J|)$, the exponents $|I|-d_{J,N}, |J|-d_{J,N}$ in \equ{last step1} are non-negative integers.  
So, in view of \equ{union1} and \equ{last step1}, the set
${A_x^N\left(t,{r},S\right)}$ can be covered with at most
\begin{align*}
& \sum_{  J \subset \{1,\dots,N\}}   k_2^{d'_{J,N}+1}\left({(1+C_0) {\frac{c_2}{c_1}\left( {\frac \theta r + 8\sqrt p }   \right)^p} k_1} \right) ^{d_{J,N}+1}{e^{\delta Nt}}
  a_1^{|I|-d_{J,N}} 
   a_2^ {|J|} \\
& \underset{\equ{lastre1}}\le  {e^{\delta Nt}} \sum_{  J \subset \{1,\dots,N\}}   k_2^{d_{J,N}+2}\left({(1+C_0) {\frac{c_2}{c_1}\left( {\frac \theta r + 8\sqrt p }   \right)^p} k_1} \right) ^{d_{J,N}+1} a_1^{|I|-d_{J,N}}  a_2^ {|J|}\\
 & \le k_3 {e^{\delta Nt}}\sum_{ J \subset \{1,\dots,N\}}   a_1^{|I|-d_{J,N}}  a_2^ {|J|}  k_3^{d_{J,N}} \\
 & =k_3 {e^{\delta Nt}} \sum_{ J \subset \{1,\dots,N\}}   a_1^{N-|J|-d_{J,N}}  a_2^ {|J|-d_{J,N}}  {({k_3a_2 }) ^{d_{J,N}}}\end{align*}
Bowen $(Nt,\theta)$-{boxes} in $P$,  where $k_3:={(1+C_0) {\frac{c_2}{c_1}\left( {\frac \theta r + 8\sqrt p }   \right)^p} k_1} k_2^2$.

\ignore{
 \\ 
&  \stackrel{\small{(1)}}\le  \frac{C_1}{\theta^{2d}} {e}^{mn(m+n)Nkt} \cdot \left( D_1+D_2+ \sqrt{\frac{C_1 \cdot D_2}{\theta^{p}} } \right)^N\\
& \underset{\equ{C1},\, \equ{C2}}{=} \frac{C_1}{\theta^{2d}} {e}^{mn(m+n)Nkt}  \left( 1-  K_1 \mu ({{{\sigma _{2 \sqrt{L}{\theta}}}{ U}}})+\frac{K_2{e}^{-\lambda kt}}{r^{mn}}+(k-1)C_{1}^k  {e^{-\frac{t}{2}}  }   + \sqrt{\frac{(k-1)C_1C_{1}^{k}}{\theta^{p}}}e^{-\frac{t}{4}} \right)^N \\
& \le \frac{C_1}{\theta^{2d}} {e}^{mn(m+n)Nkt} \left(1-  K_1 \mu ({{{\sigma _{2 \sqrt{L}{\theta}}}{ U}}})+\frac{K_2{e}^{-\lambda kt}}{r^{mn}} + {\frac{k-1}{\theta^{p}}}C_{3}^k  {e}^{-\frac{t}{4}} \right)^N 
 where $C_1={C_\alpha^42^{3p}  }$ and $C_{3}=2C_{1} \cdot \max (1,C_0)$
 and inequality $(1)$ above results from the following lemma.}
 
 \smallskip
 
To simplify the last expression we will use  the following

 \begin{lem}\cite[Lemma 5.4]{KMi2}  For any $n_1,n_2,n_3 > 0$ it holds that
 $$\sum_{  J \subset \{1,\dots,N\}}   n_1^{N-|J|-d_{J,N}}  n_2^ {|J|-d_{J,N}} n_3^{2 d_{J,N}} 
 \le   \left( n_1+n_2+ n_3 \right)^N.$$
 \end{lem}

 Applying the above lemma with $n_1=a_1$, $n_2=a_2$ and $n_3=\sqrt{k_3 a_3}$, we conclude that 
 ${A_x^N\left(t,{r},S\right)}$ can be covered with at most
$$
 k_3 {e^{\delta Nt}}  \left( a_1+a_2+ \sqrt{k_3a_3 } \right)^N 
 $$
Bowen $(Nt,\theta)$-{boxes} in $P$.
The  proof {of 
{Proposition \ref{abstract}}}  is now complete.
\end{proof}

\section{
Proof of Theorem \ref{first} 
 }\label{endofproof}
Given $P\subset G$ satisfying (ENDP), $0 < c < 1$, and $t>0$, let us define the {compact subset $Q_{c,t}$ of $X$} as follows:
\eq{qt}{Q_{c,t}:=X_{\le   C^3  \ell_{c,t}^2},}
where $\ell_{c,t}$ is as in \equ{lt} and $C$ is as in \equ{reg1}.

\ignore{Define $r_2:=\min(r_0,r_1)$ }

\smallskip
 \begin{lem}\label{first1} 
Let 
$P$ {be} a subgroup of $G$ that has {properties} {\rm (EEP)} and {\rm (ENDP)}. 
Then there exist constants
$$a',b', C_1, C_2, \lambda > 0$$
such that for any open subset $O$ of $X$ and all   {$N\in\N$}  the following holds:
For  all $0 < c < 1$ there exists $t_c>0$ such that for all $t \in \N t_c$, $0<r<1$, and $2 \le k \in \N$ satisfying
\eq{r estimate 4}{{ {e^{ \frac{a' - kt}{b'}}
} \le r  < \frac{1}{4}\min\big( r_0 \left({ \partial_{1}} Q_{c,t}  \right), {r_*}\big)},}
all {$\theta\in {\left[r ,\frac {r_*}{2}\right]}$}, and for all {$x \in \partial_r \left(Q_{c,t} \cap O^c \right)$}, the set ${A_x^N \left(kt,r,O^c\right)}$ can be covered with at most 
$$\frac{C_1}{\theta^{2p}} {e^{\delta Nkt}} \left( 1 - \mu \big( {\sigma_{4 \theta}} O  \big) +\frac{C_2}{r^{p}} e^{-\lambda kt}  + {{\frac{8C_1}{\theta^{p}}} {{\frac{\sqrt c}{1-c}}}} \right)^N    $$
Bowen $(Nkt,\theta) $-boxes
in $P$.
\ignore{\item
There exists a function $C: X \rightarrow \R^+$ such that for all $0<r<1$, all $0<s<1$, all $t \ge t_0 $, and for all $x \in X$, the set ${A}(kt,{r},Q_{c,t}^c,{N},x)$ can be covered with at most 
$$\frac{C(x)}{\theta^{p}}(k-1)^N{C_{1}}^{kN} t^{kN} e^{(mn(m+n)Nk-\frac{N}{2})t}  $
cubes of {side-length}  $\theta e^{-(m+n)Nkt}$ in $ H$.}
\end{lem}


\begin{proof}
Let $0 < c < 1$, take $t_c$   as in Corollary \ref{fin cor}, and let $0<r<1$, $2 \le k \in \N$, $t \in \N t_c$ be such that 
\equ{r estimate 4} is satisfied, where {$a',b' 
$} are as in Theorem \ref{main cor}.
Also let $\theta\in {\left[r ,\frac {r_*}{2}\right]}$. Note that the second inequality in  \equ{r estimate 4}, together with the fact that $r_0\big(\partial_1(O^c \cap Q_{c,t})\big) \ge r_0(\partial_1 Q_{c,t})$, implies {condition \equ{r estimate 3} with {$O$ replaced by $O \cup Q_{c,t}^c$}. Moreover, condition {\equ{t estimate 2}} with $t$ replaced by $kt$} follows from the first inequality in \equ{r estimate 4} .
Hence, by applying Theorem \ref{main cor} with $O$ replaced with $O \cup Q_{c,t}^c$ and $t$ replaced with $kt$, we get that for all $x \in \partial_r(O^c \cap Q_{c,t}) $ and for all $N \in \N$, the set $A_x^N \left(kt,r,O^c \cap Q_{c,t}\right)$ can be covered with at most $k_1 e^{\delta N kt}a_1^{N}$ Bowen $(Nkt,\theta)$-boxes in $P$, where
\eq{a1}{k_1= \frac{c_2}{c_1} \left(\frac{2r}{\theta}\right)^p, \quad a_1=1 - \mu \big( {\sigma_{4 \theta}} O  \big) + \frac{C_2}{r^{p}} e^{-\lambda kt}   }
and $C_2, \lambda$ are as in Theorem \ref{main cor}.

Moreover, in view of \equ{qt} and \equ{lt}, for any $x \in \partial_\theta (O^c \cap Q_{c,t}) \subset \partial_2 Q_{c,t}$ we have
$$  \frac{ \max(u(x),d)}{\ell_{c,t}^2}       \le \frac{\max(C^4\ell_{c,t}^2,d)}{\ell_{c,t}^2}  \underset{\ell_{c,t}^2 \ge \ell_{c,t}>d}{= } C^4.        $$
Also, note that
{\eq{ktlb}{kt \underset{\equ{r estimate 4}}{\ge} a'+b'\log \frac{1}{r} >b'\log \frac{1}{r}> b'\ge \frac{\log (8 \sqrt{p})}{\lambda_{\min}}.}}
Thus, by applying Corollary \ref{fin cor} we get that for all $x \in \partial_\theta (O^c \cap Q_{c,t})$ and for all $N \in \N$, the set $A_x^N \left(kt,\theta, Q_{c,t}^c\right)$ can be covered with at most $k_2e^{\delta N kt}a_2^{N}$ Bowen $(Nkt,\theta)$-boxes in $P$, where
\eq{a2}{k_2=\frac{C^4}{\nu \left({V_\theta}\right)},\, \quad a_2=\frac{4c}{1-c}. }
Now we put together the estimates we found to get an estimate for the number of Bowen $(Nkt,\theta)$-boxes needed to cover the set $A_x^N \left(kt,r,O^c \right)$. Observe that in view of \equ{ktlb}, we have {$kt \ge \frac{\log (8 \sqrt{p})}{\lambda_{\min}} $.}
So, we can apply Proposition \ref{abstract} with $S=O^c,\,Q=Q_{c,t} $  and $kt$ in place of $t$, and conclude that the set $A_x^N \left(kt,r,O^c \right)$ can be covered with at most $k_3e^{\delta Nkt} a_3^{N}$ Bowen $(Nkt,\theta)$-boxes in $P$, where $k_3, a_3$ are as in \equ{a3}, $k_1,a_1$ are as in \equ{a1}, and $k_2, a_2$ are as in \equ{a2}.

Finally, we need to estimate $k_3e^{\delta Nkt} a_3^N$ from above. We have
\eq{k3est}{
\begin{aligned}
& \, k_3  \underset{\equ{a3}} {=}{(1+C_0) {\frac{c_2}{c_1} \left( {\frac \theta r + 8\sqrt p }   \right)^p} k_1} k_2^2 \\
& \, \underset{\equ{a1},\, \equ{a2}}{=}{(1+C_0) {\left(\frac{c_2}{c_1}\right)^2 \left( {\frac \theta r + 8\sqrt p }   \right)^p}  \left(\frac{2r}{\theta}\right)^p} \left(\frac{C^4}{\nu \left({V_\theta}\right)}\right)^2 \\
&  \quad \underset{(\theta \ge r)}\le {(1+C_0) {\left(\frac{c_2}{c_1}\right)^2 \left( {2 + 16 \sqrt p }   \right)^p}} \left(\frac{C^4}{\nu \left({V_\theta}\right)}\right)^2                                       \\
& \quad \, \underset{\equ{lb1}}{\le } (1+C_0) \left(\frac{c_2}{c_1}\right)^2 \left( {2 + 16 \sqrt p }   \right)^p \left( \frac{(4 \sqrt{p})^{p}}{c_1 \theta^{p}} C^4 \right)^2 
= \frac{C_1^2}{\theta^{2p}},
\end{aligned}
}
where $C_1:=\sqrt{1+C_0} \frac{c_2}{c_1} \left( {2 + 16 \sqrt p }   \right)^{p/2}  \frac{(4 \sqrt{p})^{p}}{c_1 } C^4  \ge 1$. Furthermore, we have
\eq{a3est}
{\begin{aligned}
& a_3  \underset{\equ{a3}}{=}a_1+a_2+ \sqrt{k_3a_2}
\underset{\equ{k3est}}\le  a_1+a_2+ \sqrt{\frac{C_1^2}{\theta^{2p}} \cdot a_2}\\
&  \underset{\equ{a1},\,\equ{a2}}{=}1 - \mu \big( {\sigma_{4 \theta}} O  \big) + \frac{C_2}{r^{p}} e^{-\lambda kt}+ \frac{4c}{1-c}+ \sqrt{\frac{C_1^2}{\theta^{2p}} \cdot\frac{4c}{1-c}} \\
& \, \quad \le 1 - \mu \big( {\sigma_{4 \theta}} O  \big) + \frac{C_2}{r^{p}} e^{-\lambda kt}+ \frac{8C_1}{\theta^p} \cdot \frac{\sqrt{c}}{1-c}.
\end{aligned}
}
Therefore, by combining \equ{k3est} and \equ{a3est} we obtain
$$k_3 e^{\delta Nkt} a_3^N \le \frac{C_1}{\theta^{2p}} {e^{\delta Nkt}} \left( 1 - \mu \big( {\sigma_{4 \theta}} O  \big) +\frac{C_2}{r^{p}} e^{-\lambda kt}  + {{\frac{8C_1}{\theta^{p}}} {{\frac{\sqrt c}{1-c}}}} \right)^N $$
This ends the proof of the lemma.
\end{proof}

\begin{proof}[Proof of Theorem \ref{first}]
 Let $0 < c < 1$.  Take $t = t_c$ as in Lemma \ref{first1}, and let $Q = Q_{c,t_c}$ 
{be} as in  \equ{qt}. Also 
let $O$ be an open subset of $X$. \medskip

\noindent\textbf{Proof of (1):}
{Take $2 \le k \in \N$ {and $x\in X$}.  Our goal is to find an upper bound for the Hausdorff dimension of the set {$S(k,t,x) $ defined in  \equ{S1}. In view of \equ{S1} and
the countable stability of Hausdorff dimension  it suffices to 
estimate the  dimension of 
$$\left \{h \in \overline{V_{{r_*}/{2}}}: {{g_{Nkt}}}hx \notin  Q \,\,\, \forall N \in \N       \right \} 
,$$
which, due to \equ{qt}, coincides with $\bigcap_{N \in \N} {A_x^N \big(kt,\frac{r_*}{2},X_{>  C^3  \ell_{c,t}^2}\big)}$.}

From Corollary \ref{fin cor} applied with $\theta= \frac{r_*}{2}$, combined with Lemma \ref{coveringballs} applied with $t$ replaced {by} $Nkt$ and ${r}=\frac{r_*}{2}$, we get {that} for any $N \in \N$ the set {$A_x^N{\big(kt,\frac{r_*}{2},X_{>  C^3  \ell_{c,t}^2}\big)}$} can be covered with at most
{$$      {{\frac{ e^{p {\lambda_{\max}}Nkt} }{{\nu({V_{ {r_*}/{2}}}) } }} }   \left({\frac{4c}{1-c}}\right)^N  \frac{\max\big(u(x),d\big)}{\ell_{c,t}^2}$$}
balls of radius $\frac{r_*}{2} e^{ - {\lambda_{\max}} Nkt}$ in $P$. Hence,

{ 
$${\begin{aligned}
    & \dim   \bigcap_{N \in \N} {A_x^N\left(kt,\frac{r_*}{2},X_{>  C^3  \ell_{c,t}^2}\right)}  \\
 & \le \lim_{N \rightarrow \infty}  \frac{ \log \left(  {{\frac{ e^{p {\lambda_{\max}}Nkt} }{{\nu({V_{ {r_*}/{2}}}) } }} } ({\frac{4c}{1-c}})^N 
 \frac{\max\big(u(x),d\big)}{\ell_{c,t}^2} \right)}{- \log  {\frac{r_*}{2}}{e}^{-{\lambda_{\max}}Nk{t}}} \\
 & = \frac{\log {\frac{4c}{1-c}} e^{p{\lambda_{\max}} k t}}{{\lambda_{\max}}kt} 
 =  {p- { \frac{1}{{\lambda_{\max}} k t}  \log {\frac{1-c}{4c}}}. }
\end{aligned}}$$}
}
\medskip

\noindent \textbf{Proof of (2):} { Let $2 \le k \in \N$, and $x \in X$.} Our goal is to find an upper bound for the Hausdorff dimension of the set 
$${\left\{h\in P \ssm S(k,t,x) : hx\in \widetilde {E}({F^+},O) \right\}}
$$
{Recall that
$${S(k,t,x)^c =  \left\{h \in P:{{g_{Nkt}}}hx \in  Q \text{ for some } N \in \N    \right\}}.
$$
Therefore}
{
$${\begin{aligned}
 \left\{h\in P \ssm S(k,t,x) : hx\in \widetilde{E}({F^+},O) \right\}   =& \left \{h \in P: hx \in  \widetilde{E}({F^+},O) \bigcap \left(\bigcup_{N \in \N}  g_{-Nkt}    Q \right)\right \}   \\ \subset & \left \{ h\in P: hx \in \bigcup_{N \in \N}  g_{-Nkt}   \big(Q \cap \widetilde{E}({F^+},O)  \big) \right \} \\
= & \bigcup_{N \in \N}  \left \{ h\in P: hx \in   g_{-Nkt}   \big(Q \cap \widetilde{E}({F^+},O)  \big) \right \}.
\end{aligned}}$$}
Hence, since Hausdorff dimension is {countably} stable, to complete the proof of this part, it suffices to show that for any $N \in \N$ we have
\eq{linequality}{\dim \left \{ h\in P: hx \in   g_{-Nkt}   \big(Q \cap \widetilde{E}({F^+},O)  \big) \right \} \le p- { \frac{ \mu \big( {{\sigma_{4 \theta}}} O  \big) -\frac{C_2}{r^{p}} e^{-\lambda kt} - { {{\frac{8C_1}{\theta^{p}}} {{\frac{\sqrt c}{1-c}}}}} }{{\lambda_{\max}} kt}}  }
where $C_1, C_2, \lambda$ as in Lemma \ref{first1}.

\smallskip

Now let $N \in \N$
and suppose $r>0$ {is} such that \equ{ineq beta2} is satisfied, where $a',b'$ are as in Lemma \ref{first1}. Note that, since $P$ is normalized by $F^+,$ we have $P=g_{-Nkt}Pg_{Nkt}$. Moreover, {$V_{r}$} is a tessellation domain. Hence, by countable stability of Hausdorff dimension, in order to prove \equ{linequality},  it suffices to show that the Hausdorff dimension of the set 
$$E'_{N,x,r}:=  \left\{h \in g_{-Nkt} \, {\overline{V_{r}}} \, g_{Nkt} : hx \in g_{-Nkt}  \big(Q \cap \widetilde{E}( {F^+},O)  \big) \right\}$$
is not greater than the right-hand {side} of  \equ{linequality}.
{For any $h \in E'_{N,x,r}$ we have
$$
\begin{aligned}
g_{ikt} g_{Nkt}hx 
& = g_{ikt} (g_{Nkt}h g_{-Nkt})g_{Nkt}x \in O^c                  \quad\forall\,i\in\N,\\
\end{aligned}
$$
and at the same time $g_{Nkt}h g_{-Nkt} \in {\overline{V_r}}.$
Hence,
 \eq{fininc}{E'_{N,x,r} \subset g_{-Nkt} \left( \bigcap_{i \in \N} {A_{g_{Nkt}x}^i \left( kt,r,O^c \right)} \right)g_{Nkt} .}}
Also, it is easy to see that if $E'_{N,x,r}$ is non-empty, then $$g_{Nkt}x \in {\overline{V_{r}}}  \left(Q \cap O^c \right) \subset \partial_{r} \left(Q \cap O^c \right).$$ So by applying  Lemma  \ref{coveringballs} {with $r$ replaced by $\theta$ and $t$ replaced with $ikt$, and Lemma \ref{first1} with $t$ replaced {by} $kt$}, we get that for any $i \in \N$ and any $\theta\in {\left[r,\frac {r_*}{2}\right]}$, the set  {$A_{g_{Nkt}x}^i \left( kt,r,O^c \right) $} can be covered with at most
$$
{
\begin{aligned}
& {\frac{
C_1 }{ \theta^{2p}}   e^{p {\lambda_{\max}} ikt}   \left( 1 - \mu \big( {\sigma_{4 \theta}} O  \big) +\frac{C_2}{r^{p}} e^{-\lambda ikt}  +  {{\frac{8C_1}{\theta^{p}}} {{\frac{\sqrt c}{1-c}}}} \right)^i}
\end{aligned}}  
$$
balls of radius $\theta e^{ -  {\lambda_{\max}} ik t}$ in $P$.
Also, note that {the \hd\ is preserved by conjugation.}
So, {we have for any $\theta\in {\left[r ,\frac {r_*}{2}\right]}$:
$${\begin{aligned} 
 \dim E'_{N,x,r}  & \underset{\equ{fininc}}\le  \dim \left(g_{-Nkt} \left( \bigcap_{i \in \N} {A_{g_{Nkt}x}^i \left( kt,r,O^c \right)} \right)g_{Nkt}  \right)\\
& 
{=}  \dim \bigcap_{i \in \N} {A_{g_{Nkt}x}^i \left( kt,{r},O^c \right)}  \\
& \le \lim_{i \rightarrow \infty} \frac {\log \left({\frac{
C_1 }{ \theta^{2p}} \cdot e^{p {\lambda_{\max}} ikt} \cdot \left( 1 - \mu \big( {\sigma_{4 \theta}} O  \big) +\frac{C_2}{r^{p}} e^{-\lambda ikt}  +  {{\frac{8C_1}{\theta^{p}}} {{\frac{\sqrt c}{1-c}}}} \right)^i}  \right)}{{- \log {\theta}{e}^{-{\lambda_{\max}} ikt}}} \\
& = p - \frac{-\log \left(   1 - \mu \big( {\sigma_{4 \theta}} O  \big) +\frac{C_2}{r^{p}} e^{-\lambda kt} +  {{\frac{8C_1}{\theta^{p}}} {{\frac{\sqrt c}{1-c}}}}\right)}{ {\lambda_{\max}}kt}\\
& \le {p - \frac{ \mu \big( {\sigma_{4 \theta}} O  \big) -\frac{C_2}{r^{p}} e^{-\lambda kt} - { {{\frac{8C_1}{\theta^{p}}} {{\frac{\sqrt c}{1-c}}}}} }{{\lambda_{\max}} kt}}. 
\end{aligned}}$$}
  This finishes the proof.
\end{proof}

\section{Concluding Remarks}\label{remarks}

\subsection{Effective estimates} {It is a natural problem to effectivize the estimates showing up in the Dimension Drop Conjecture. 
Previous work of the authors on the subject \cite{KMi1, KMi2} contained explicit estimates, although with no claims of optimality. Namely, this has been done  under the assumption that the complement of $O$ is compact (in particular, when $X$ is compact), and also in the special  case \equ{gt}.

 In the more general set-up of this paper it is also possible to make the estimates effective.
  This however would require an additional ingredient: {finding a lower bound for the injectivity radii of compact sets $\{ x: u_t(x) \le M   \}$ arising from condition (ENDP). Such lower bounds can be obtained immediately whenever the following condition is satisfied: Let $P$ be a subgroup of $G$ that has property (ENDP), and let $\{u_t\}_{t \ge t_0}$ be the family of height functions as in Definition \ref{enp2}; then  there exist positive constants $\m_0, m$ such that
\eq{bic} {{r_0 (x)}^{-1} \ge m_0 {u_t(x)}^{-m} \,\,\, \text{ for every   } x \in X,\, t \ge t_0.             }}
This condition can be verified in many special cases. {For example, in \cite{SS, BQ} certain 
height functions are constructed on homogeneous spaces of semisimple Lie groups without compact factors, and for these height functions \equ{bic} is verified in \cite[Proposition 26]{SS} and \cite[Lemma 6.3]{BQ}  respectively. By using the same method one can easily show that \equ{bic} holds} for height functions $u_t$ as in Theorem \ref {thmguanshi}, and also for the family of height functions constructed in \cite{KMi2} in the case  \equ{gt}. A variation of our argument shows that in the presence of \equ{bic} one has
$${\inf_{x \in X} \codim \left(\big\{h\in P: hx\in \widetilde E({F},O)\big\} \right) \gg \frac{\mu(O)}{\log \frac{1}{{\min \left(\theta_O, \mu(O), r_1     \right)}}},}$$
where 
$\theta_O$ is as in \equ{su1}, and $0<r_1<\frac{1}{2}$ is a uniform constant independent of $O$.}

\subsection{Removing the $\Ad$-diagonalizability condition} We expect that
by a slight modification of the proof of Theorem \ref{dimension drop 3}  one can show that this theorem holds when $F$ is an arbitrary one-parameter unbounded subsemigroup of a connected semisimple Lie group $G$; namely, the condition that $F$ is $\Ad$-diagonalizable is not necessary. Indeed, recall the Jordan decomposition of $F = \{g_t\}$: one can write $g_t=k_ta_tu_t$, where
 $K_{F} = \{k_t\}$ is bounded,  $A_{F} = \{a_t\}$  is $\Ad$-diagonalizable, and $U_{F} = \{u_t
\}$  is $\Ad$-unipotent.   These subgroups are uniquely determined  and commute with each other. 
If $A_{F}$ is trivial (in other words, if $F$ is $\Ad$-quasiunipotent) and $U_F$ is not, then   Ratner's Measure Classification Theorem and the work of Dani and Margulis  (see \cite[Lemma 21.2]{St} and  \cite[Proposition 2.1]{DM}}) imply that whenever $O$ is non-empty, the set $\widetilde E({F},O)$  is contained in a countable union of proper submanifolds of $X$; hence dimension drop takes place in a stronger form. 
On the other hand, if $A_{F}$ is non-trivial, one can modify our argument following the lines of \cite[\S4]{GS}, where an analog of   (ENDP) was 
considered with $(I_{f,t}\psi)(x)$ as in 
\equ{ift} replaced by a family of operators 
$$\psi(\cdot) \mapsto \int_P
f({{h}})\psi(a_tgu_tg^{-1}{{h}}\,\cdot)\,d{{\nu}}({{h}}),$$
and with $g$ running through the centralizer of $A_F$ in $G$.

{\subsection{Jointly Dirichlet-improvable systems of linear forms: a dimension bound} 
{Fix $m,n\in\N$ and, given $c\le 1$, say that $Y\in \mr$ is \textsl{$c$-Dirichlet improvable} 
if 
for all sufficiently large $N$ 
\eq{di}{\begin{aligned}\text{ there
exists $\p\in\Z^m$ and $\vq\in\Z^n\nz$  such that }\\\|Y\vq - \p\| < cN^{-n/m} 
\text{ and }0  <  \|\vq\|  < N.\qquad \end{aligned}}
(In this subsection 
$\|\cdot\|$ stands for the supremum norm on $\R^m$, $\R^n$ and $\R^{m+n}$.)
We let 
$\mathbf{DI}_{m,n}(c)$ be the set of  $c$-Dirichlet improvable $Y\in\mr$.
Dirichlet's theorem (see e.g.\ \cite{S2}) implies that $\mathbf{DI}_{m,n}(1) = \mr$.
Davenport  and
Schmidt \cite{DS}  proved that the Lebesgue measure of $\mathbf{DI}_{m,n}(c)$  is zero for any $c < 1$, and also that $\bigcup_{c<1}\mathbf{DI}_{m,n}(c)$  contains the set 
of badly approximable $m\times n$ matrices, which is known \cite{S1} to have full \hd; in other words, $ \dim \mathbf{DI}(c) \to mn$ as $c\to 1$.}}

Recently in \cite{KMi2} a solution of DDC for the case \equ{gt}, that is for the space $X$ of unimodular lattices in $\R^{m+n}$, was used to derive a dimension drop result for    the family $\{\mathbf{DI}_{m,n}(c)\}$: namely, that $ \dim \mathbf{DI}_{m,n}(c) < mn            $ whenever  $c<1$. Moreover, as explained in \cite[Remark 6]{KSY}, a combination of the methods from \cite{KMi2} with measure estimates obtained in \cite{KSY} can produce an effective estimate for the codimension of $\mathbf{DI}_{m,n}(c)$.
The reduction to dynamics goes back to Davenport, Schmidt and Dani \cite{dani}. It proceeds by assigning an element $h_Y :=\begin{bmatrix}
I_m & Y \\
0 & I_n\end{bmatrix}$ of $G = \SL_{m+n}(\R)$ to $Y$. Arguing as in \cite[Proposition 2.1]{KW1} or \cite[Proof of Theorem 1.5]{KMi2},  one can see that $Y\in \mathbf{DI}_{m,n}(c)$ if and only if $h_Y\Z^{m+n}\in \widetilde E(F,O)$, where  \eq{defo}{ O = \left\{\Lambda\in X: \|\vv\| \ge c^{\frac m{m+n}}\text{ for all }\vv\in\Lambda\nz\right\}}
(a subset of $X$ with non-empty interior), and $X$, $F$ are as in \equ{gt}. 


Our new Diophantine appplication  is motivated by \cite[\S 2.7]{BV}, where Beresnevich and Velani introduced the notion of \textsl{jointly singular} $k$-tuples of matrices. 
Namely, say that $(Y_1,\dots,Y_k)\in M_{n,m}^k$  is {\sl  $c$-Dirichlet improvable} if 
for all sufficiently large $N$ 
\eq{jdi}{\begin{aligned}&\text{ there
exist 
 $\p \in\Z^m,\, \vq\in\Z^n\nz$  and $i \in\{1, \dots,k\}$  such that }\\
& \qquad \qquad \qquad |Y_i \vq - \p\| < cN^{-n/m} 
\text{ and }0  <  \|\vq\|  < N. \end{aligned}}
Denote the set of   $c$-Dirichlet improvable $(Y_1,\dots,Y_k)$ by   $\mathbf{DI}_{m,n}^{(k)}(c)$. 
Applying  Dirichlet's theorem for each $k$, it is easy to see that $\mathbf{DI}_{m,n}^{(k)}(1) = M_{n,m}^k$. 
When $c<1$ one wants   for each large $N$ to  improve the conclusion of Dirichlet's theorem for at least one of the matrices, and for different $N$ it does not have to be the same matrix. It is clearly if one of the matrices is itself $c$-Dirichlet improvable, then so is the whole $k$-tuple; however  in general    $\mathbf{DI}_{m,n}^{(k)}(c)$ could be much larger than  the set  $$\big\{(Y_1,\dots,Y_k)\in M_{n,m}^k: Y_i \in \mathbf{DI}_{m,n} (c)\text{ for some }i=1,\dots,k\big\}.$$ This raises a problem of showing  some sort of dimension drop,
which is achieved by reducing the problem to a flow on 
the product of $k$ copies of $X$ as in \equ{gt}.
Indeed, it is not hard to see that the validity of \equ{jdi} for all sufficiently large $N$ is equivalent to the statement that for all sufficiently large $t$  
\eq{product}{
\exists\,\vv 
\in\Z^{m+n}\nz\text{ and }i \in \{1, \dots, k\}
\text{ with }\|g_th_{Y_i}\vv\| < c^{\frac m{m+n}}
.}
\ignore{$$
\mathcal{R}_c := \left\{\begin{pmatrix}a \\ \mathbf{b}\end{pmatrix}\in\R^{1+n} :|a| < c,\ \|\mathbf{b}\| \le 1\right\},
$$
where for any $s= (s_1,\cdots, s_n) \in \R^n$
$$h_s:=   \begin{bmatrix}
   1 & s_1 & \hdots & s_n \\
   0 & 1  \\
   \vdots &   & \ddots \\
   0      &   &    & 1
 \end{bmatrix}.          $$}
In its turn, 
\equ{product} is equivalent to 
$$(g_th_{Y_1}\Z^{m+n},\dots,g_th_{Y_k}\Z^{m+n})\notin O\times\cdots\times O,$$ where $O$ is as in \equ{defo}.
We conclude that $(Y_1,\dots,Y_k)\in\mathbf{DI}_{m,n}^{(k)}(c)$ if and only if 
$(h_{Y_1}\Z^{m+n},\dots,h_{Y_k}\Z^{m+n}) \in \widetilde E(F^{(k)},O\times\cdots\times O)$, where 
$$F^{(k)} := \{(g_t,\dots,g_t): t\ge 0\}   \subset \prod_{i=1}^k G$$
is acting on 
$X^{(k)} :=   \prod_{i=1}^k X$.
\ignore{the following: there exists $\vv = \begin{pmatrix}-p \\ \vq\end{pmatrix}\in\Z^{1+n}\nz$ such that the vector $(g_th_{\xi_1} \vv, \cdots ,g_th_{\xi_m}\vv)$ in $\prod_{
i=1}^m X$  belongs to the complement of the set  $\prod_{
i=1}^n U_c$, where
$$U_c := \big\{x\in X : x\cap \mathcal{R}_c = \{0\}\big\}.$$
It is easy to see that $U_c$ has non-empty interior; hence, $\prod_{
i=1}^n U_c$ has non-empty interior as well. Thus, for any matrix $\mathbf{\Xi}$ in $M_{n,m}$, $\mathbf{\Xi} \in \mathbf{JDI}(c)$  is equivalent to $(h_{\xi_1}, \cdots, h_{\xi_m}) \in E(D^+,\prod_{
i=1}^m U_c )$, where $\xi_1, \cdots, \xi_m$ are columns of $\mathbf{\Xi}$ and $D$ is the one-parameter subgroup $\{(g_t,\cdots,g_t): t \in \R\}$ of $\prod_{
i=1}^m \SL_{n+1}(\R)$. }

Since $F^{(k)}$ is a diagonalizable subsemigroup of $\prod_{
i=1}^k G$ whose expanding horospherical subgroup is precisely $$H^{(k)}:= \prod_{i=1}^k \{h_{Y}: Y \in \mr\},$$ 
it follows from  Theorem \ref{thmguanshi} that  $H^{(k)}$  has property {(ENDP)} with respect to \linebreak $(X^{(k)}, F^{(k)})$.
Moreover, since the action of $F$ on $X$ is exponentially mixing, by using Fubini's Theorem  it is straightforward to check that the action of $F^{(k)}$ on $X^{(k)}$ is exponentially mixing as well; hence, by Theorem \ref{thmheep}, $H^{(k)}$ has property {(ENDP)} with respect to $(X^{(k)}, F^{(k)})$. Therefore, we can apply Theorem \ref{dimension drop 3} with $P=H^{(k)}$ and arrive at:

\begin{thm} \label{jdi}
The \hd\ of $\mathbf{DI}_{m,n}^{(k)}(c)$ is strictly less than $kmn$ for any $c<1$ and $k\in \N$.
\end{thm}

\bibliographystyle{alpha}

\begin{thebibliography}{AAEKMU}


\bibitem[AAEKMU]{MRC}H.\ al Saqban, P.\ Apisa, A.\ Erchenko, O.\ Khalil, S.\ Mirzadeh, and C.\ Uyanik, \textsl{Exceptional directions for the Teichm\"{u}ller geodesic flow and Hausdorff dimension}, J.\ Eur.\ Math.\ Soc.\ {\bf 23} (2021), 1423--1476.


\bibitem[AGMS]{AGMS} J.\ An, L.\ Guan, A.\ Marnat and R.\ Shi, \textsl{Divergent trajectories on products of homogeneous spaces}, 
Adv.\ Math.\ {\bf 390} (2021), Paper No.\ 107910.

\bibitem[BQ]{BQ}Y.\ Benoist and J.\ Quint, \textsl{Mesures stationnaires et fermés invariants des espaces homogénes}, Ann.\ of Math.\ {\bf 174} (2011), no. 2, 1111--1162.















 \bibitem[BV]{BV} V.\ Beresnevich and S.\ Velani,  \textsl{Number theory meets wireless communications: an introduction for dummies like us}, in: {\bf Number theory meets wireless communications}, pp.\ 1--67, Math.\ Eng., Springer, Cham, 2020.
\bibitem[C]{C} Y.\ Cheung,   \textsl{Hausdorff dimension of the set of singular pairs}, Ann.\ of Math.\ {\bf 173} (2011),
127--167.
\bibitem[Da]{dani}  S.\,G.\ Dani,
\textsl{Divergent trajectories of flows on
homogeneous spaces and Diophantine approximation},
J.\ Reine Angew.\ Math.\ {\bf 359} (1985), 55--89.




\bibitem[DM]{DM}
S.\,G.\ Dani and G.\,A.\ Margulis,
\textsl{Limit distributions of orbits of unipotent flows and values of
  quadratic forms},
in {\bf I.\,M.\ Gelfand Seminar}, pp.\ 91--137,   Adv.\ Soviet Math., {vol.\ 16}, Part 1, Amer.\ Math.\
Soc.,
  Providence, RI, 1993.
  
  
\bibitem[DS]{DS}  H.\ Davenport and W.\,M.\ Schmidt,
\textit{Dirichlet's theorem on diophantine approximation}, in: Symposia
Mathematica, Vol.\ IV (INDAM, Rome, 1968/69),  1970.


 \bibitem[EKP]{EKP}  M.\ Einsiedler, S.\ Kadyrov and A.\ Pohl, \textsl{Escape of mass and entropy for diagonal flows in
real rank one situations}, Israel J.\ Math.\ {\bf 210} (2015), no.\ 1, 245--295.  
\bibitem[EMM]{EMM}A.\ Eskin, G.A.\ Margulis and S.\ Mozes, \textsl{Upper bounds and asymptotics in a
quantitative version of the Oppenheim conjecture}, Ann.\ Math.\ 147 (1998), no.\
2, 93--141.


\bibitem[EM]{EM}A.\ Eskin and  G,A.\ Margulis, \textsl{Recurrence properties of random walks on finite volume homogeneous manifolds}, In: {\bf Random Walks and Geometry}, pp.\ 431--444, de Gruiter, Berlin (2004).



\bibitem[EMo]{EMo}
  A.~Eskin and S.~Mozes, \textsl{Margulis functions and their applications},
  in:
{\bf Dynamics, Geometry, Number Theory: The Impact of Margulis on Modern Mathematics.}  
{D.~Fisher, D.~Kleinbock, and G.~Soifer, eds.},
University of Chicago Press, 2022.


  
\bibitem[GS]{GS}
L.\ Guan and R.\ Shi, \textsl{Hausdorff dimension of divergent trajectories on homogeneous spaces}, 
{Compositio Math.\  {\bf 156} (2020), no.\ 2, 340--359.} 



\bibitem[Ka]{K} S.\ Kadyrov, \textsl{Exceptional sets in homogeneous spaces
and Hausdorff dimension}, Dyn.\ Syst.\ {\bf 30} (2015), no.\ 2, 149--157.









\bibitem[KKLM]{KKLM}   S.\ Kadyrov, D.\ Kleinbock, E.\ Lindenstrauss  and G.\,A.\ Margulis, \textsl{Singular systems of linear forms 
and non-escape of mass in the space of lattices},   J.\ Anal.\ Math.\ {\bf 133} (2017),  253--277. 

\bibitem[KM1]{KM1} D.\ Kleinbock and G.\,A.\ Margulis, \textsl{Bounded orbits of nonquasiunipotent flows on homogeneous spaces}, in: {\bf Sina\'i's Moscow Seminar on Dynamical Systems}, 141--172, Amer.\ Math.\ Soc.\ Trans.\ Ser.\ 2, Amer.\ Math.\ Soc.\, Providence, RI, 1996.

{\bibitem[KM2]{KM4} \bysame, \textsl{On effective equidistribution of expanding translates of certain orbits in the space of
lattices}, in: {\bf Number Theory, Analysis and Geometry}, Springer, New York, 2012, pp.\ 385--396.}






\bibitem[KMi1]{KMi1} D.\ Kleinbock and S.\ Mirzadeh, \textsl{Dimension estimates for the set of points with non-dense orbit in
homogeneous spaces}, Math.\ Z.\ {\bf 295} (2020), 1355--1383.

\bibitem[KMi2]{KMi2} \bysame,  \textsl{On the dimension drop conjecture for diagonal flows on the space of lattices}, 
arXiv preprint 
{\tt arXiv:2010.14065}  (2020)
to appear in Adv.\ Math.







\bibitem[KS]{KS}
 D.\ Kelmer and P.\ Sarnak, \textsl{Strong spectral gaps for compact quotients of products of }$\PSL(2,\R)$, J.\ Eur.\ Math.\
Soc.\ {\bf 11} (2009), no.\ 2, 283--313.


\bibitem[KSY]{KSY} {D.\ Kleinbock, A.\ Str\"ombergsson, and S.\ Yu},  \textsl{A measure estimate in geometry of numbers and improvements to Dirichlet's theorem}, 
{Proc.\ London Math.\ Soc.}, published online at {\tt DOI:10.1112/plms.12470} (2022).


\bibitem[KW1]{KW1}
D.\ Kleinbock and B.\ Weiss, \textsl{Dirichlet's theorem on Diophantine approximation
and homogeneous flows},  J.\ Mod.\ Dyn. {\bf 4} (2008), 43--62.

\bibitem[KW2]{KW2} \bysame, 
\textsl{Modified Schmidt games and  a conjecture of Margulis}, J.\ Mod.\ Dyn.\  {\bf 7}, no.\ 3 (2013), 429--460.





\bibitem[M]{M} G.\,A.\ Margulis,   {\bf On some aspects of the theory of Anosov systems}, Springer Monographs in Mathematics, Springer-Verlag, Berlin, 2004. 


\bibitem[RHW]{RHW} F.\ Rodriguez Hertz and Z.\ Wang, \textsl{On $\varepsilon$-escaping trajectories in homogeneous spaces}, Discrete Contin.\ Dyn.\ Syst.\ {\bf 41} (2021), no.\ 1, 329--357.

\bibitem[SS]{SS} A.\ Sanchez and J.\ Seong, \textsl{An avoidance principle and Margulis functions for expanding translates of unipotent orbits}, arXiv preprint {\tt arXiv:2206.12019}  (2022). 

\bibitem[Sc1]{S1} W.\ Schmidt,    \textsl{Badly approximable systems of
linear forms}, J.\ Number Theory {\bf 1} (1969), 139--154.
\bibitem[Sc]{S2} \bysame,  \textbf{\da},  Lecture Notes in Mathematics, vol.\ 785, Springer-Verlag, Berlin, 1980.






\bibitem[St]{St}  A.\ Starkov, \textbf{Dynamical systems on homogeneous spaces}, Translations of Mathematical Monographs, {\bf 190}, American Mathematical Society, Providence, RI, 2000.

\bibitem[We]{Weg} H.\ Wegmann,  \textsl{Die Hausdorff-Dimension von
kartesischen Produktmengen in metrischen R\" aumen},  J.\ Reine
Angew.\ Math.\ {\bf 234} (1969), 163--171.








\end{thebibliography}

\end{document}